\documentclass[11pt, oneside, reqno]{article}
\usepackage{fullpage}
\usepackage{amssymb,latexsym,amsmath,verbatim,layout,hyperref,amsthm,color,authblk}  
\numberwithin{equation}{section}
\usepackage[T1]{fontenc}
\usepackage {mathrsfs}                
  
\theoremstyle{mytheoremstyle}

\newtheorem{thm}{Theorem}[section]
\newtheorem*{thm*}{Theorem}
\newtheorem{lem}[thm]{Lemma}
\newtheorem*{lem*}{Lemma}
\newtheorem{prop}[thm]{Proposition}
\newtheorem*{prop*}{Proposition}

\newtheorem*{cor}{Corollary}

\theoremstyle{definition}
\newtheorem{defn}[thm]{Definition}

\theoremstyle{remark}
\newtheorem*{rmk}{Remark}

\usepackage{hyperref}

\renewcommand{\vec}[1]{\mathbf{#1}}

\newcommand{\vgrad}{\mathbf{\nabla}}     
\newcommand{\vlap}{\mathbf{\Delta}}     
\newcommand{\sph}{\mathbb{S}}
\newcommand{\R}{\mathbb{R}}
\renewcommand{\P}{\mathbb{P}}
\newcommand{\E}{\mathbb{E}}
\newcommand{\N}{\mathbb{N}}

\title{Strong solutions for the Stochastic Navier-Stokes equations on the 2D rotating sphere with stable L\'evy noise}

\author{Leanne Dong$^1$}
\date{%
    $^1$Faculty of Engineering and IT, The University of Technology Sydney\\%
    \today
}

\begin{document}

	\maketitle
\begin{abstract}
In this paper we prove the existence and uniqueness of a strong solution (in PDE sense) to the stochastic Navier-Stokes equations on the rotating 2-dimensional unit sphere perturbed by stable L\'evy noise. This strong solution turns out to exist globally in time.
\end{abstract}	
   
\section{Introduction}
The deterministic Navier-Stokes system (NSEs) on the rotating sphere serves as a basic model in large scale ocean dynamics. Many authors have studied the NSEs on the unit spheres. Notably, Il'in and Filatov \cite{MR953596,MR1055527} tackled the well-posedness of these equations and identified the Hausdorff dimension of their global attractors \cite{MR1211366}. Teman and Wang investigated the inertial forms of NSEs on the sphere while Teman and Ziane show that the NSEs on a 2D sphere is a limit of NSE defined on a spherical cell \cite{MR1472301}. 
Our paper is concerned with the following stochastic Navier-Stokes equations (SNSEs) on a 2D rotating sphere:          
\begin{align}\label{snse1}
    \partial_t u+\nabla_{ u} u-\nu \mathbf{L}u+\mathbf{\omega}\times u+\nabla p= f+\eta(x,t),\quad \text{div } u=0,\quad u(0)=u_0,
\end{align}    
where $\mathbf{L}$ is the stress tensor, $\mathbf{\omega}$ is the Coriolis acceleration, $ f$ is the external force and $\eta$ is the noise process that can be informally described as the derivative of an $H$-valued L\'evy process. Rigorous definitions of all relevant quantities in this equation will be given in sections 2 and 3. To the best of our knowledge, there are only three papers which discuss stochastic Navier-Stokes equations on spheres\cite{MR3306386,BGQ17, varner2015unique}. All these were concerned with the Gaussian case. In particular, the authors in \cite{MR3306386} proved the existence and uniqueness of weak solutions to (\ref{snse1}) with additive Gaussian noise. Moreover, they proved that the associated random dynamical system is asymptotically compact, which induced the existence of a compact random attractor and the existence of an invariant measure in their accompanying paper \cite{BGQ17}. The author in \cite{varner2015unique} studied the Navier-Stokes system on spheres with a Gaussian kick force and a deterministic force. The main contribution was the existence and uniqueness of a time-invariant measure.

Our paper is the first paper to discuss SNSEs on the sphere with a stable L\'evy noise. There are three new features which distinguish our paper from other work in the literature on SNSEs on spheres and SNSEs with L\'evy noise. First, the domain of consideration is a sphere. Next, the noise is of a stable type which is ruled out by many existing studies on stochastic PDEs with L\'evy noise. Third, our well-posedness result is new in the sense of a strong solution in PDE sense.

The aim of our paper is to prove the existence and uniqueness of a global strong solution to (\ref{snse1}). In particular, we prove that
given a $\mathbb{L}^4$-valued noise, $H$-valued forcing $f$ and small $V$-valued initial data, there exists a unique global strong solution in PDE sense for the abstract stochastic Navier-Stokes equations on the 2D unit sphere perturbed by stable L\'evy noise, which depends continuously on the initial data. 
The time interval of existence depends on the regularity of the forcing and the assumptions imposed on noise.

The paper is organised as follows. In section \ref{sec:nse2ds}, we review the fundamental mathematical theory for the deterministic Navier-Stokes equations (NSEs) on the sphere. We state some known results without proofs. In section \ref{sec:snsesph}, we define the SNSEs on spheres. We start with some analytic facts; we introduce the driving noise process, which is a stable L\'evy noise via subordination.  The SNSEs are then decomposed into an Ornstein Uhlenbeck (OU) process (associated with the linear part of the SNSEs) and nonlinear PDEs. 
In section \ref{sec:num45}, we prove strong classical solution (see the proof of Theorem 3.3.7) for smooth initial data, sufficient regular noise following the classical lines in the proof of Theorem 3.1 \cite{MR1207308}. 

\section{Navier-Stokes equations on a rotating 2D unit sphere}
\label{sec:nse2ds}
The sphere is the simplest example of a compact Riemannian manifold without boundaries, hence one may employ the well-developed tools from Riemannian geometry to study objects on such a manifold. Nevertheless, all objects of interest in this thesis are defined explicitly under the spherical coordinates. The presentation here follows closely from Goldys et al. \cite{MR3306386} and references therein.
\subsection{Preliminaries}\label{ssec:pre}
Let $\mathbb{S}^2$ be a 2D unit sphere in $\R^3$, that is $\mathbb{S}^2=\{x=(x_1,x_2,x_3)\in\R^3: |x|=1\}$. An arbitrary point $x$ on $\mathbb{S}^2$ can be parametrized in the spherical coordinates as
\begin{align*}
     x=\hat{ x}(\theta,\phi)=(\sin\theta\cos\phi,\sin\theta\sin\phi,\cos\theta),\quad 0\le\theta\le\pi,\quad 0\le\phi\le 2\pi.
\end{align*}
The corresponding angle $\theta$ and $\phi$ will be denoted by $\theta( x)$ and $\phi( x),$ or simply by $\theta$ and $\phi.$

Let $e_{\theta}=e_{\theta}(\theta,\phi)$ and $e_{\phi}=e_{\phi}(\theta,\phi)$
be the standard unit tangent vectors of $\sph^2$ at point $\hat{ x}(\theta,\phi)\in\sph^2$ in the spherical coordinates, that is,
\begin{align*}  e_{\theta}&=(\cos\theta\cos\phi,\cos\theta\sin\phi,-\sin\theta),\quad e_{\phi}=(-\sin\phi,\cos\phi,0).
\end{align*}
We remark that
\begin{align*}
    e_{\theta}&=\frac{\partial\hat{ x}(\theta,\phi)}{\partial\theta},\quad e_{\phi}= \frac{1}{\sin\theta}\frac{\partial\hat{ x}(\theta,\phi)}{\partial\phi},
\end{align*}
where the second identity holds whenever $\sin\theta\neq 0$.

Our first objective is to give a meaning to all of the terms in the deterministic Navier-Stokes equations for the velocity field $ u(\hat{ x},t)=(u_{\theta}(\hat{ x},t), u_{\phi}(\hat{ x},t))$ of a geophysical fluid flow on the 2D rotating unit sphere $\sph^2$ under the external force $         f=(f_{\theta},f_{\phi})=f_{\theta}e_{\theta}+f_{\phi}e_{\phi}$. The motion of the fluid is governed by the equation 
\begin{align}\label{NSE4}   
    \partial_t  u+\vgrad_{ u} u-\nu\mathbf{L} u+\omega\times  u+\frac{1}{\rho}\vgrad p=  f,\quad\text{div}\,\, u=0,\,\,  u( x,0)= u_0.
\end{align}
Here $\nu$ and $\rho$ are two positive constants denoting the viscosity and the density of the fluid; the normal vector field
\begin{align*}
    \omega=2\Omega\cos(\theta( x)) x,
\end{align*}
where $ x=\hat{x}(\theta( x),\phi( x))$ ; $\Omega$ is the angular velocity of the Earth; and $\theta$ is the parameter representing the colatitude. Note that $\theta( x)=\cos^{-1}(x_3).$ In what follows we will identify $\mathbf{\omega}$ with the corresponding scalar function $\omega$ defined by $\omega( x)=2\Omega\cos(\theta( x)).$ We will introduce now the other terms that appear in the equation. The surface gradient for a scalar function $f$ on $\sph^2$ is given by
\begin{align*}
    \nabla f=\frac{\partial f}{\partial\theta}e_{\theta}+\frac{1}{\sin\theta}\frac{\partial f}{\partial\phi}e_{\phi},\quad 0\le\theta\le\pi,\,\,0\le\phi\le 2\pi.
\end{align*}
Unless specified otherwise, by a vector field on $\mathbb{S}^2$ we mean a tangential vector field. That is, a section of the tangent vector bundle of $\sph^2.$

On the other hand, for a vector field $ u=(u_{\theta}, u_{\phi})$ on $\sph^2$, that is $ u=u_{\theta}e_{\theta}+u_{\phi}e_{\phi}$, one puts
\begin{align}\label{nsep}
    \text{div}\, u=\frac{1}{\sin\theta}\left(\frac{\partial}{\partial\theta}(u_{\theta}\sin\theta)+\frac{\partial}{\partial\phi}u_{\phi}\right).
\end{align}
Given two vector fields $ u$ and $ v$ on $\sph^2$, there exist vector fields $\tilde{ u}$ and $\tilde{ v}$
defined in some neighbourhood of the surface $\sph^2$ and such that their restrictions to $\sph^2$ are equal to $ u$ and $ v$. More precisely, see Definition 3.31 in \cite{MR2090752}, 
\[\left.\tilde{ u}\right|_{\sph^2}= u:\sph^2\to T\sph^2,\quad\mathrm{and}\quad \left.\tilde{ v}\right|_{\sph^2}= v:\sph^2\to T\sph^2\,.\]
For $ x\in\R^3$, we define the orthogonal projection $\pi_x:\R^3\to T_x\sph^2$  of $x$ onto $T_x\sph^2$, that is
\begin{align}
    \pi_x: \R^3\ni  y\,\mapsto\, y-( x\cdot y) x=- x\times( x\times y)\in T_{ x}\sph^2.
\end{align}
\begin{lem}[\cite{BGQ17}]
    Suppose $\tilde{ u}$ and $\tilde{ v}$ are $\R^3$-valued vector fields on $\sph^2$, and $ u$, $ v$ are tangent vector fields on $\sph^2$, defined by $ u( x)=\pi_{ x}(\tilde{ u}(x))$ and $ v( x)=\pi_{ x}(\tilde{ v}(x))$
, $ x\in\sph^2$. Then the following identity holds: 
\begin{align}\label{ocrossu}
    \pi_{ x}(\tilde{ u}( x)\times\tilde{ v}( x))= u( x)\times(( x\cdot v( x)) x)+(( x\cdot u( x)) x\times v( x),\quad  x\in\sph^2.
\end{align}
\end{lem}
\begin{proof}
 Let us fix $ x\in\sph^2.$ Then one may decompose vectors $\tilde{ u}$ and $\tilde{ v}$ into tangential and normal components as follows: 
    \begin{align*}
    \tilde{ u}= u+ u^{\perp}\quad\text{with}\quad  u\in T_x\sph^2,\quad  u^{\perp}=( u\cdot x) x    ,
    \end{align*}
    \begin{align*}
    \tilde{ v}= v+ v^{\perp}\quad\text{with}\quad  v\in T_x\sph^2,\quad  v^{\perp}=( v\cdot x) x    .
    \end{align*}
Since $ u\times  v$ is normal to $T_x\sph^2$, $\pi_x( u\times v)=0.$ Likewise, $ u^{\perp}\times  v^{\perp}=0$ since the cross-product of two parallel vectors yields the 0 vector. Hence, it follows that
\begin{align}\label{ocross}
    \pi_{ x}(\tilde{ u}\times\tilde{ v})=\pi_{ x}( u\times v+ u\times v^{\perp}+ u^{\perp}\times  v)=u\times v^{\perp}+ u^{\perp}\times  v.
\end{align}
\end{proof}

We will denote by $\tilde\nabla$ the usual gradient in $\R^3$ and then we have 
\begin{align}
    (\nabla f)(x)=\pi_x(\tilde{\nabla}\tilde{f}( x)).
\end{align}
The operator $\text{curl}$ is defined by the formula 
\begin{align}
    (\text{curl}\,u)( x)=(I-\pi_x)((\tilde\nabla\times\tilde{ u})( x))=( x\cdot(\tilde\nabla\times\tilde{ u})( x)) x.
\end{align}
Let $ u$ be a tangent vector field on $\sph^2$. Applying formula (\ref{ocross}) to the vector fields $\tilde{ u}$ and $\tilde{ v}=\tilde{\nabla}\times\tilde{ u}$, one gets
\begin{align}\label{ocrosscur}
\pi_{ x}(\tilde{ u}\times(\tilde{\nabla}\times\tilde{ u}))&=
\tilde{ u}\times(\tilde{\nabla}\times( u^{\perp}+ u)\notag\\
&= u\times ((\nabla\times  u)^{\perp})+ u^{\perp}\times(\nabla\times  u)\notag\\
&= u\times(( x\cdot(\tilde{\nabla}\times\tilde{ u})) x)\notag\\
&=( x\cdot(\tilde{\nabla}\times\tilde{ u}))( u\times x),\quad x\in\sph^2.
\end{align}
So, we can now define the curl of the vector field $ u$ on $\sph^2$, namely,
\begin{align}\label{curu}
    \text{curl }u :=\hat{ x}\cdot (\tilde{\nabla}\times\tilde{u})|_{\sph^2}.     
\end{align}     
 Equations (\ref{curu}) and (\ref{ocrossu}) together yield 
\begin{align*}
\pi_{ x}[\tilde{ u}\times(\tilde{\nabla}\times\tilde{ u})](x)=[ u( x)\times x]\,\text{curl} u(x),\quad  x\in \sph^2.
\end{align*}
Therefore, we have the following:
\begin{defn}
    Let $u$ be a tangent vector field on $\sph^2$, and let the vector field $\psi$ be normal to $\sph^2$. We set
    \begin{align}
        \text{curl}\, u&=(\hat{ x}\cdot(\tilde{\vgrad}\times\tilde{ u})) |_{\sph^2},\quad \text{Curl}\,\psi=(\tilde{\vgrad}\times \mathbf{\psi})|_{\sph^2}.
    \end{align}
\end{defn}
\par\bigskip\noindent
The first equation above indicates a projection of $\nabla\times\tilde{u}$ onto the normal direction, while the 2nd equation means a restriction of $\nabla\times\psi$ to the tangent field on $\sph^2$. The definitions presented above do not depend on the extensions $\tilde{u}$ and $\tilde{\psi}.$ A vector field $\psi$ normal to $\sph^2$ will often be identified with a scalar function on $\sph^2$ when it is convenient to do so. The following expressions describe the relationships among Curl of a scalar function $\psi$, Curl of a normal vector field $ \vec{w} =w\hat{ x}$, and curl of a vector field $ v$ on $\sph^2$.
\begin{align}\label{3curl}
    \text{Curl}\,\psi&=-\hat{ x}\times\nabla\psi,\quad \text{Curl}\, \vec{w} =-\hat{ x}\times\nabla w ,\quad \text{curl } v=-\text{div}(\hat{ x}\times v).
\end{align}
Let 
\begin{align}
    (\vgrad_{ v} u)( x)&=\pi_{ x}\left(\sum^3_{i=1}\tilde{ v}_i( x)\partial_i\tilde{ u}( x)\right)=\pi_{ x}\left((\tilde{v}(x)\cdot\tilde{\nabla})\tilde{ u}( x)\right),\quad  x\in\sph^2.
\end{align}
Invoking \eqref{ocrossu} and the formula \begin{align*}
    (\tilde{ u}\cdot\tilde\nabla)\tilde{ u}=\tilde{\nabla}\frac{|\tilde{ u}^2|}{2}-\tilde{ u}\times(\tilde\nabla\times\tilde{ u}),
\end{align*}
we find that the covariant derivative $\vgrad_{ u} u$ takes the form 
\begin{align*}
    \vgrad_{ u} u=\nabla\frac{|{ u}^2|}{2}-\pi_x(\tilde{ u}\times(\tilde\nabla\times\tilde{ u})).
\end{align*}
In particular, using \eqref{ocrossu} we obtain 
\begin{align*}
    \vec{\nabla}_{ u} u=\nabla\frac{| u|^2}{2}- \pi_x(\tilde{ u}\times(\tilde\nabla\times\tilde{ u})).
\end{align*}
 The surface diffusion operator acting on vector fields on $\sph^2$ is denoted by $\mathbf{\Delta}$ (known as the Laplace de Rham operator) and is defined as
\begin{align}\label{surdiffuse}
    \mathbf{\Delta} v=\nabla\text{div } v-\text{Curl curl } v.
\end{align}
Using (\ref{3curl}), one can derive the following relations connecting the above operators:
\begin{align}\label{relateop}
    \text{div Curl }v&=0,\quad\text{curl Curl v}=-\hat{x}\Delta v,\quad \mathbf{\Delta}\text{Curl } v=\text{Curl}\Delta v.
\end{align}
Next, we recall the definition of the Ricci tensor Ric of the 2D sphere $\sph^2$. Since 
\[
    \text{Ric}=
    \begin{pmatrix}
    E & F\\ F & C,
\end{pmatrix}\]
where the coefficients $E,F,G$ of the first fundamental form are given by 
\begin{align*}
    E&= x_{\theta}\cdot  x_{\theta}=1;\\
    F&= x_{\theta}\cdot  x_{\phi}=x_{\phi}\cdot x_{\theta}=0;\\
    C&= x_{\phi}\cdot x_{\phi}=\sin^2\theta,
\end{align*}
we find that 
\begin{align}\label{eq_ric}
    \text{Ric}=
\begin{pmatrix}
    1 & 0\\ 0 & \sin^2\theta
\end{pmatrix}\,.
\end{align}
Finally we define the stress tensor $\mathbf{L}$. It is given by
\begin{align*}
    \mathbf{L}=\vlap+2\text{Ric},
\end{align*}
where $\vlap$ is the Laplace-de Rham operator.  \subsection{Function spaces on the sphere}\label{ssec:func}
In what follows we denote by $dS$ the surface measure on $\sph^2$. In the spherical coordinates one has locally, $dS=\sin\theta\, d\theta d\phi.$ For $p\in[1,\infty)$, we denote by $L^p=L^p(\sph^2,\R)$ of $p$-integrable scalar function on $\sph^2$, endowed with the norm
\begin{align*}
 |v|_{L^p}=\left(\int_{\sph^2}|v( x)|^p dS( x)\right)^{1/p}.
\end{align*}
For $p=2$ the corresponding inner product is denoted by
\begin{align*} (v_1,v_2)=(v_1,v_2)_{L^2(\sph^2)}=\int_{\sph^2}v_1 v_2\,dS.
\end{align*}
On the other hand, we denote by $\mathbb{L}^p=\mathbb{L}^p(\sph^2)$ the space $L^p(\sph^2, T\sph^2)$ of vector fields $ v:\sph^2\to T\sph^2$ endowed with the norm
\begin{align*}
    | v|_{L^p}=\left(\int_{\sph^2}| v( x)|^p dS( x)\right)^{1/p},
\end{align*}
where, for $ x\in\sph^2$, $| v( x)|$ denotes the length of $ v( x)$ in the tangent space $T_{ x}\sph^2.$ For $p=2$ the corresponding inner product is denoted by
\begin{align*}
    ( v_1, v_2)=( v_1, v_2)_{\mathbb{L}^2}=\int_{\sph^2} v_1\cdot v_2\, dS.
\end{align*}
In this paper, the induced norm on $\mathbb{L}^2(\mathbb{S}^2)$ is denoted by $|\cdot|$. For other inner product spaces, say $V$ with the inner product $(\cdot,\cdot)_V$, the associated norm is denoted by $|\cdot|_V.$

The following identities hold for appropriate real valued scalar functions and vector fields on $\sph^2$, see (2.4)-(2.6) \cite{MR1055527}:
\begin{align}
    (\nabla\psi, v)&=-(\psi,\text{div}\, v),\label{gradpsi}\\
    (\text{Curl }\psi, v)&=(\psi,\text{curl } v),\label{Curlpsi}\\
    (\text{Curlcurl }w ,z)&=(\text{curl }w ,\text{curl }z).\label{Ccurl}
\end{align}
In (\ref{Curlpsi}), the $\mathbb{L}^2(\sph^2)$ inner product is used on the left hand side while the $L^2(\sph^2)$ is used on the right hand side. Throughout this paper, we identify a normal vector field $ \vec{w} $ with a scalar field $w$ and by $ \vec{w} =\hat{ x}w$. We hence put
\begin{align}
    (\psi, \vec{w} ) &:=(\psi,w)_{L^2(\sph^2)},\quad \text{if}\,\, \vec{w} =\hat{x}w,\quad \psi,w\in L^2(\sph^2).
\end{align}
Let us now introduce the Sobolev spaces $H^1(\sph^2)$ and $\mathbb{H}^1(\sph^2)$ of scalar functions and vector fields on $\sph^2$. 
Let $\psi$ be a scalar function and let  $ u$ be a vector field on $\sph^2$, respectively. For $s\ge 0$ we define
\begin{align}   |\psi|^2_{H^1(\sph^2)}=|\psi|^2_{L^2(\sph^2)}+|\nabla\psi|^2_{L^2(\sph^2)},
\end{align}
and
\begin{align}\label{h1norm}
    | u|^2_{\mathbb{H}^1(\sph^2)}&=| u|^2+|\nabla\cdot u|^2+|\text{Curl}\, u|^2\,.
\end{align}    
One has the following Poincar\'e inequality
\begin{align}\label{eq_poinc}
    \lambda_1| u|^2\le |\text{div}\, u|^2+|\text{Curl}\,u|^2,\quad u\in\mathbb{H}^1(\sph^2),
\end{align}
where $\lambda_1>0$ is the first positive eigenvalue of the Laplace-Hodge operator, see below. 
By the Hodge decomposition theorem in Riemannian geometry \cite{MR972259}, the space of $C^{\infty}$ smooth vector field on $\mathbb{S}^2$ can be decomposed into three components:
\begin{align*} C^{\infty}(T\mathbb{S}^2)=\mathcal{G}\oplus\mathcal{V}\oplus\mathcal{H},
\end{align*}
where
\begin{align*}
    \mathcal{G}=\{\nabla\psi\in C^{\infty}(\mathbb{S}^2)\},\quad \mathcal{V}=\{\text{Curl}\psi\in C^{\infty}(\mathbb{S}^2)\},
\end{align*}
and $\mathcal{H}$ is the finite-dimensional space of harmonic vector fields. Since the sphere is simply connected, that is, the map $\mathbb{S}^2\to\mathbb{S}^2$ is a diffeomorphism, we have  $\mathcal{H}=\{0\}.$ The condition of orthogonality to $\mathcal{H}$ is dropped out. We introduce the following spaces:
\begin{align}\label{abstractH}
    H&:=\{ u\in\mathbb{L}^2(\mathbb{S}^2): \nabla\cdot  u=0\},\\   V&:=H\cap\mathbb{H}^1(\mathbb{S}^2)\notag.
\end{align}
In other words, $H$ is the closure of the
\begin{align*}
    \{ u\in C^{\infty}(T\mathbb{S}^2):\nabla\cdot  u=0\}
\end{align*}
in the $\mathbb{L}^2$ norm $| u|=( u, u)^{1/2}$, where $ u=( u_{\theta}, u_{\phi})$ and
\begin{align}
    ( u, v)=\int_{\mathbb{S}^2} u_{\theta} v_{\theta}+u_{\phi} v_{\phi}d S(x).
\end{align}
The space $V$ is the closure of
 \begin{align*}
     \{ u\in C^{\infty}(T\mathbb{S}^2):\nabla\cdot u=0\}
 \end{align*}
 in the norm of $\mathbb H^1\left(\mathbb S^2\right)$. 
Since $V$ is densely and continuously embedded into $H$, and $H$ can be identified with its dual $H'$, one has the following Gelfand triple:
\begin{align}
    V\subset H\cong H'\subset V'.
\end{align}
\subsection{Stokes operator}\label{ssec:weakformu}
We will recall first that the Laplace-Beltrami operator on $\mathbb S^2$
\begin{align}\label{LB}
	\Delta f = \frac{1}{\sin\theta}\frac{\partial}{\partial\theta}\left(\sin\theta\frac{\partial f}{\partial\theta}\right)+\frac{1}{\sin^2\theta}\frac{\partial^2 f}{\partial\phi^2}
\end{align}
 can be defined in terms of spherical harmonics $Y_{l,m}$ as follows (See also \cite{MR1022665}).  For $\theta\in [0,\pi]$, $\phi\in [0,2\pi)$, we define 
\begin{align} Y_{l,m}(\theta,\varphi)=\left[\frac{(2l+1)(l-|m|)!}{4\pi(l+|m|)!}\right]^{1/2}P^m_l(\cos\theta)e^{im\varphi},\quad m=-l,\cdots,l,
\end{align}     
with $P^m_l$ being the associated Legendre polynomials. The family $\left\{Y_{l,m}:\, l=0,1,\ldots,\,\,m=-l,\ldots,l\right\}$ form an orthonormal basis in $L^2\left(\mathbb S^2\right)$
and we then can define the well known Laplace-Beltrami operator on $\sph^2$ (\ref{LB})
by putting 
	\[\Delta Y_{l,m}=-l(l+1)Y_{l,m}.\]
Then one can extend by linearity to all functions $f:L^2\left(\mathbb S^2\right)$ such that 
	\[\sum_{l=0}^\infty\sum_{m=-l}^ll^2(l+1)^2\left(f,Y_{l,m}\right)^2_{L^2\left(\mathbb S^2\right)}<\infty\,.\]

We consider the following linear Stokes problem \cite{MR3306386}. That is, given $f\in V'$, find $v\in V$ such that
\begin{align}
    \nu\text{Curlcurl} u-2\nu\text{Ric}( u)+\nabla p=f,\quad\text{div}\,u=0.
\end{align}    
By taking the inner product of the first equation above with a test field $v\in V$, and then using (\ref{Ccurl}), the pressure term drops and we obtain
\begin{align*}
    \nu(\text{curl}\,u,\text{curl}\, v)-2\nu(\text{Ric}\,u, v)=(  f,v)\quad\forall\, v\in V.
\end{align*}
Without loss of generality, let $\nu = 1$, we define a bilinear form $a: V\times V\to\R$ by
\begin{align}\label{bf1}
	a(u,v)=(-Lu,v).
\end{align}
By performing some elementary calculations, one can write (\ref{bf1}) as follows:
\begin{align}\label{bf2}
    a( u, v) :=(\text{curl}\, u,\text{curl}\, v)-2(\text{Ric}\, u, v),\quad
     u,\, v\in V.
\end{align}
In view of (\ref{h1norm}) and formula \eqref{eq_ric} for the Ricci tensor on $\sph^2$, the bilinear form $a$ satisfies
\begin{align}
    a( u, v)\le| u|_{\mathbb{H}^1}| v|_{\mathbb{H}^1}
\end{align}
and so it is continuous on $V$. So, by the Riesz representation theorem, there exists a unique operator $\mathcal{A}: V\to V'$ where $V'$ is the dual of $V$, such that $a( u, v)=(\mathcal{A} u, v)$, for $ \{u, v\}\in V.$ Let us recall that by the results in \cite{MR1187618}, p.1446, we also have
\begin{align*}
	a(u,u)=|\text{Def}\,u|^2_2,\quad u\in V
\end{align*}     
where $\text{Def}$ is the deformation tensor (See \cite{MR1187618} for more details). Then by the Poincar\'e inequality \eqref{eq_poinc} we find that $a( u, u)\ge\alpha| u|^2_{V}$, for a certain $\alpha>0$, which implies that $a$ is coercive in $V$. Hence, by the Lax-Milgram theorem, the operator $\mathcal{A}:V\to V'$ is an isomorphism. Let  $ A $ be a restriction of $\mathcal {A}$ to $H$:
\begin{align}\label{domA}      
    \begin{cases}
        D( A )& :=\{u\in V :\,\mathcal Au\in H\}, \\
         A u& :=\mathcal{A}u,\quad u\in D( A ).
    \end{cases}
\end{align}
It is well known  (see for instance \cite{tanabe1979equations}, Theorem 2.2.3 ) that $A$ is positive definite, self-adjoint in $H$ and $D(A^{1/2})=V$ with equivalent norms. 
 Furthermore, for some positive constants $c_1,c_2$ we have 
 \[c_1|u|_{D(A)}\le|Au|\le c_2|u|_{D(A)}\,,\]
 \begin{align}\label{Anorm}
     \langle   A u, u\rangle=(( u, u))=| u|_V=|\nabla  u|^2=|Du|^2,\quad  u\in D(A).
 \end{align}
The spectrum of A consists of an infinite sequence of eigenvalues $\lambda_l$. Using the stream function $\psi_l$ for which $ w_l=\text{Curl}\psi_{l,m}$ and identities (\ref{relateop}), one can show that each $\lambda_l$ is in fact the vector of eigenvalues of the Laplace-Beltrami operator $\Delta$. That is $\lambda_l =l(l+1)$. Additionally, there exists an orthonormal basis $\left(\vec{Z}_{l,m}\right)_{l\ge 1}$ of $H$ consisting of the  eigenvector of $ A $, where 
\begin{align}    \vec{Z}_{l,m}=\lambda^{-1/2}_l\text{Curl}Y_{l,m},\quad l=1,\ldots,m=-l,\ldots, l.
\end{align}
Therefore, for any $v\in H$, one has,
\begin{align}   v=\sum^{\infty}_{l=1}\sum^l_{m=-l}\hat{v}_{l,m}\mathbf{Z}_{l,m},\quad \widehat{v}_{l,m}=\int_{\sph^2}v\cdot\vec{Z}_{l,m}dS=(v,\vec{Z}_{l,m}).
\end{align}
      
An equivalent definition of the operator $A$ can be given using the so-called Leray-Helmhotz projection $P$ that is defined 
as an orthogonal projection from $\mathbb{L}^2(\sph^2)$ onto $H$. Let $\mathbb{H}^2(\sph^2)$ denote the domain of the Laplace-Hodge operator in $H$ endowed with the graph norm. 
It can be shown from \cite{MR2569498} that $D( A )=\mathbb{H}^2(\sph^2)\cap V$ and $ A =-P(\vlap+2\text{Ric})$. Therefore, we obtain an equivalent definition of the Stokes operator on the sphere. 
\begin{defn}
    The Stokes operator $A$ on the sphere is defined as
    \begin{align}\label{optA}
         A : D(A)\subset H\to H, \quad  A =-P(\mathbf{\Delta}+2\text{Ric}),\quad D( A )=\mathbb{H}^2(\mathbb{S}^2)\cap V,
    \end{align}

where $\mathbf\Delta$ is the Laplace-De Rham operator. 
\end{defn}
It can be shown that $V=D( A ^{1/2})$ when endowed with the norm $|x|_V=| A ^{1/2}x|$ and the inner product $((x,y))=\langle  A x,y\rangle$. After identification of $H$ with its dual space we have $V\subset H\subset V'$ with continuous dense injection. The dual pairing between $V$ and $V'$ is denoted by $(\cdot,\cdot)_{V\times V'}.$
Moreover, there exist positive constants $c_1$, $c_2$ such that 
\begin{align*}
    c_1|u|^2_V\le (Au, u)\le c_2 |u|^2_V,\quad u\in D(A).
\end{align*}

Let us now introduce the Sobolev spaces $H^s(\sph^2)$ and $\mathbb{H}^2(\sph^2)$ of scalar functions and vector fields on $\sph^2$. 
Let $\psi$ be a scalar function and let  $ u$ be a vector field on $\sph^2$, respectively. For $s\ge 0$ we define
\begin{align}  |\psi|^2_{H^s(\sph^2)}=|\psi|^2_{L^2(\sph^2)}+|(-\Delta)^{s/2}\psi|^2_{L^2(\sph^2)},
\end{align}
and
\begin{align}
    | u|^2_{\mathbb{H}^s(\sph^2)}=| u|^2+|(-\vlap)^{s/2} u|^2,
\end{align}
where $\Delta$ is the Laplace-Beltrami operator and $\vlap$ is the Laplace-de Rham operator on the sphere.
Note that, for $k=0,1,2,\cdots$ and $\theta\in (0,1)$ the space $H^{k+\theta}(\sph^2)$ can be defined as the interpolation space between $H^k(\sph^2)$ and $H^{k+1}(\sph^2)$. One can apply the same procedure for $H^{k+\theta}(\sph^2)$, due to\cite{BGQ17}. 
The fractional power $A^{s/2}$ of the Stokes operator $A$ in $H$ for any $s\ge 0$ is given by
\begin{align*}
    D(A^{s/2})&=\left\{v\in H: v= \sum^{\infty}_{l=1}\sum^{l}_{m=-l}\hat{v}_{l,m}\vec{Z}_{l,m},\,\,\sum^{\infty}_{l=1}\sum^{l}_{m=-l}\lambda^s_{l}|\hat{v}_{l,m}|^2<\infty\right\},\\
    A^{s/2}v&:=\sum^{\infty}_{m=1}\sum^{l}_{m=-l}\lambda^{s/2}_{l}\hat{v}_{l,m}\vec{Z}_{l,m}\in H.
\end{align*} 
The Coriolis operator $ \mathbf{C}_1 :\mathbb{L}^2(\mathbb{S}^2)\to\mathbb{L}^2(\mathbb{S}^2)$ is defined by the formula\footnote{The angular velocity vector of Earth is denoted by $\Omega$ consistant with geophysical fluid dynamics literature. It shall not be confused with the notation for probility space $\Omega$ used in this paper.}
\begin{align}
    ( \mathbf{C}_1 v)( x)=2\Omega( x\times  v( x))\text{cos}\theta,\quad  x\in\mathbb{S}^2.
\end{align}
It is clear from the above definition that $ \mathbf{C}_1$ is a bounded linear operator defined on $\mathbb{L}^2(\mathbb{S}^2).$ In what follows we will need the operator $\mathbf{C}=P \mathbf{C}_1$ which is well defined and bounded in $H.$ Furthermore, for $u\in H$,
\begin{align}\label{corio}
    ( \mathbf{C} u, u)=( \mathbf{C}_1  u,P u)=\int_{\mathbb{S}^2}2\Omega\text{cos}\theta(( x\times  u)\cdot u( x))dS( x)=0.
\end{align}
In addition,
\begin{lem}\label{corioinner}
    For any smooth function $ u$ and $s\ge 0$
    \begin{align}
        ( \mathbf{C} u,  A ^s  u)=0.
    \end{align}
\end{lem}
\begin{proof}
    The case $s=0$ is obvious as in the line above, due to the fact that $(\omega\times  u)\cdot  u=0$. For $s>0$ we refer readers to Lemma 5 in \cite{MR2869773}.
\end{proof}
\par\bigskip\noindent
Let $X=H\cap\mathbb L^4\left(\mathbb S^2\right)$ be endowed with the norm 
\[|v|_X=|v|_H+|v|_{\mathbb L^4\left(\mathbb S^2\right)}.\]
Then $X$ is a Banach space. It is known that the Stokes operator $A$ generates an analytic $C_0$-semigroup $\{e^{-tA}\}_{t\ge 0}$ in $X$ (see Theorem A.1 in \cite{MR3306386}). Since the Coriolis operator $\mathbf{C}$ is bounded on $X$, we can define in $X$ an operator
\begin{align*}
    \hat{A}=\nu A+\mathbf{C},\quad D(\hat{A})=D(A),
\end{align*}
with $\nu>0.$

\begin{lem}\label{lem1.2}
    Suppose that $V\subset H\cong H'\subset V'$ is a Gelfand triple of Hilbert spaces. If a function $ u$ being $L^2(0,T;V)$ and $\partial_t  u$ belongs to $L^2(0,T;V')$ in weak sense, then $ u$ is a.e. equal to a continuous function from $[0,T]$ to $H$; the real function $| u|^2$ is absolutely continuous; and, in the weak sense one has
    \begin{align}\label{chainrule}
    \partial_t | u(t)|^2=2\langle\partial_t  u(t), u(t)\rangle.  
    \end{align}
\end{lem}

\begin{prop}\label{analytics}
    The operator $\hat{A}$ with the domain $D(\hat{A})=D(A)$ generates a strongly continuous and analytic semigroup $\{e^{-tA}\}_{t\ge 0}$ in $X$. In particular, there exist $M\ge 1$ and $\mu>0$ such that 
    \begin{align}      |e^{-t\hat{A}}|_{\mathcal{L}(X,X)}\le Me^{-\mu t},\quad t\ge 0\,;
    \end{align}
and for any $\delta>0$ there exists $M_{\delta}\ge 1$ such that
\begin{align}\label{analyticexp}
    |\hat{A}^{\delta}e^{-t\hat{A}}|_{\mathcal{L}(X,X)}\le M_{\delta} t^{-\delta}e^{-\mu t},\quad t>0.
\end{align}    
\end{prop}
\begin{proof}
See the proof of Proposition 5.3 in \cite{MR3306386}.
\end{proof}
\par\bigskip\noindent
Now consider the trilinear form $b$ on $V\times V\times V$, defined as
\begin{align}
    b( v, w, z)=(\nabla_{ v} w, z)=\int_{\mathbb{S}^2}\nabla_v w\cdot  z dS=\pi_x\sum^3_{i,j=1}\int_{\Omega}v_jD_i w_j z_jdx,\quad v,w,z\in V.
\end{align}
Using the following identity (See \cite{MR3306386}),
\begin{align*}
    2\vec{\nabla}_w v=-\text{curl}(w\times v)+\nabla(w\cdot v)-v\, \text{div }w+w\,\text{div }v-v\times\text{curl }w-w\times\text{curl }v,
\end{align*}
and equation (\ref{surdiffuse}), one can write the divergence free fields $v,w,z$ in the trilinear form as follows:
\begin{align}\label{triiden}
    b(v,w,z)=\frac{1}{2}\int_{\sph^2}[-v\times w\cdot\text{curl }z+\text{curl }v\times w\cdot z-v\times\text{curl }w\cdot z]dS.
\end{align}
Now, we know that the bilinear form $B: V\times V\to V'$ is defined by
\begin{align}\label{optB}    (B(u,v),w)=b(u,v,w)=\sum^3_{i,j=1}\int_{\Omega}u_i\frac{\partial(v_k)_j}{\partial x_i}u_j dx,\quad w\in V.
\end{align}
Moreover,
\begin{align}\label{b0}
    b( v, w, w)=0,\quad b( v, z, w)=-b( v, w, z),\quad  v\in V,  w, z\in \mathbb{H}^1(\mathbb{S}^2),
\end{align}
and such that
\begin{align}\label{b01}
    |B( u, v), w|=|b( u, v, w)|\le c| u|| w|(|\text{curl } v|_{\mathbb{L}^{\infty}(\mathbb{S}^2)}+| v|_{\mathbb{L}^{\infty}(\mathbb{S}^2)}),\quad  u\in H, v\in V, w\in H,
\end{align}

\begin{align}\label{b1}
    |B( u, v), w|=|b( u, v, w)|\le c| u|^{1/2}| u|^{1/2}_{V}| v|^{1/2}| v|^{1/2}_{V}| w|_{V},\quad  u, v, w\in V,
\end{align}

\begin{align}\label{b2}
    |B( u, v), w|&=|b( u, v, w)|\le c| u|^{1/2}| u|^{1/2}_{V}| v|^{1/2}_{V}| A  u|^{1/2}| w|,\,\, \forall\, u\in V, v\in D( A ),  w\in H,\,\, n=2,
\end{align}

\begin{align}\label{b5}
|b( u, v, w)|\le c| u|_{\mathbb{L}^4(\mathbb{S}^2)}| v|_V| w|_{\mathbb{L}^4(\mathbb{S}^2)},\quad  v\in V,  u, w\in\mathbb{H}^1(\mathbb{S}^2).
\end{align}
In view of (\ref{b1}),
\begin{align}
\sup_{z\in V, |z|_V\ne 0}\frac{|(B( u, v), z)|}{|z|_V}=|B( u, v)|_{V'}\le  c| u|^{1/2}| u|_V^{1/2}| v|^{1/2}| v|_V^{1/2}\notag\\
\Longrightarrow\,|B( u, u)|_{V'}\le  c| u|| u|_V\label{bvdual},\\
|B( u, u)|_{H}\le  c| u|| u|_V.\notag
\end{align}
\begin{align}
\sup_{z\in H, |z|_H\ne 0}\frac{|(B( u, v), z)|}{| z|_H}=|B( u, v)|_{H}\le  c| u|^{1/2}| u|_V^{1/2}| v|^{1/2}| v|_V^{1/2}\notag\\
 \Longrightarrow\,
 |B( u, u)|_{H}\le  c| u|| u|_V.\label{b3}
\end{align}
In view of (\ref{b2}),
\begin{align}
&\    \sup_{z\in H, |z|_H\ne 0}\frac{|(B( u, v), z)|}{| z|_H}=|B( u, v)|_H\le  c| u|^{1/2}| u|_V^{1/2}| u|^{1/2}| A  u|^{1/2}\notag\\
&\ \Longrightarrow\,|B( u, u)|_{H}\le  c| u|^{1/2}| u|_V| A  u|^{1/2}\le c| u|^{1/2}_V| u|_V| A  u|^{1/2} \quad\forall\,\, u\in D( A ). \label{b4}
\end{align}
In view of (\ref{b5}), $b$ is a bounded trilinear map from $\mathbb{L}^4(\mathbb{S}^2)\times V\times \mathbb{L}^4(\mathbb{S}^2)$ to $\R.$ 
\begin{lem}\label{uniquext}
    The trilinear map $b$ can be uniquely extended from $V\times V\times V$ 
to a trilinear map 
\[b:(\mathbb{L}^4(\mathbb{S}^2)\cap H)\times \mathbb{L}^4(\mathbb{S}^2)\times V\to\R\,.\]
 \end{lem}
Finally, we recall the interpolation inequality (See \cite{MR953596}, p.12),
\begin{align}\label{ladyzhen}
    | u|_{\mathbb{L}^4(\mathbb{S}^2)}\le C|u|^{1/2}_{\mathbb{L}^2(\mathbb{S}^2)}|u|^{1/2}_V\,.
\end{align}
Inequality (\ref{b1}) is deduced from
the following Sobolev embedding:
\begin{align*}
    H^{1/2}=W^{1/2,2}(\mathbb{S}^2)\hookrightarrow \mathbb{L}^4(\mathbb{S}^2).
\end{align*}
Then using (\ref{surdiffuse}), (\ref{gradpsi}), (\ref{domA}) and (\ref{triiden}), we arrive at the \emph{weak solution} of the Navier-Stokes equations (\ref{nsep}), which is a vector field $u\in L^2([0,T];V)$ with $u(0)=u_0$ that satisfies the weak form of (\ref{nsep}):
\begin{align}
    (\partial_t u,v)+b(u,u,v)+\nu (\text{curl} u,\text{curl} v)-2\nu(\text{Ric }u,v)+(\mathbf{C}u,v)=(f,v),\quad v\in V  ,  
\end{align} 
where the bilinear form is defined earlier.
 With a slight abuse of notation, we denote $B(u)=B(u,u)$ and $B(u)=\pi(u,\nabla u)$.

\section{Stochastic Navier-Stokes equations on the 2D unit sphere}\label{sec:snsesph}
By adding a L\'evy white noise to (\ref{NSE4}), we obtain the main equation in this paper: 
\begin{align}\label{SNSE4}
		\partial_t u+\nabla_u u-\nu \vec{L} u+\omega\times u+\nabla p&=f+\eta(x,t),\\    
		\text{div }u=0,\,\, u(x,0)&=u_0,\, x\in\mathbb{S}^2.\notag
\end{align}
We assume that, $u_0\in H$, $f\in V'$ and $\eta(x,t)$ is the so-called L\'evy white noise. That is, a noise process which can be informally described as the derivative of a $H$-valued L\'evy process, that is rigorously defined in Lemma \ref{leml4}. 
Applying the Leray-Helmholz projection we can interpret equation (\ref{SNSE4}) as an abstract stochastic  equation in $H$ 
\begin{align}\label{asnse4}
	du(t)+Au(t)+B(u(t),u(t))+\mathbf{C}u=fdt+GdL(t),\quad u(0)=u_0,
\end{align}

where $L$ is an $H$-valued stable L\'evy process and $G:H\to H$ is a bounded operator. In order to study this equation we need to consider first some properties of stochastic convolution. 

\subsection{Stochastic convolution of $\beta$-stable noise}\label{ssec:convolnbeta}
In this section we will recall a linear version of equation 
\eqref{asnse4}
\begin{align}\label{eq_lin1}
	dz(t)+Az(t)+\mathbf{C}z=GdL(t),\quad z(0)=0\,.
\end{align}
Under appropriate assumptions formulated below, its solution takes the form 
\begin{align}\label{eq_lin2}
	z_t=\int^t_0 e^{-\widehat{A}(t-s)}GdL(s),
\end{align}
where $\widehat A=A+\mathbf C$. Let  $W$ be a cylindrical Wiener process on a Hilbert space $K$ continuously imbedded into $H$ and let $X$ be a $\beta/2$-stable subordinator\footnote{See definition in p.50, Eg 1.3.19 in \cite{MR2512800}.}. Let us now denote the stable distribution $S_{\alpha}(\sigma,\beta,\mu)$ in consistence with page 9 in \cite{MR1280932}, where $\alpha\in(0,2]$, $\sigma\ge 0$, $\beta\in[-1,1]$, $\mu\in\R$. Then the process $L=W(X)$ is a symmetric cylindrical $\beta$-stable process in $H$.
  
We need the Ornstein–Uhlenbeck process (\ref{eq_lin2}) to take value in $X$. To this end, we need the following definition.
\begin{defn}    
	Let $K$ and $X$ be separable Banach spaces and let $\gamma_K$ be the canonical cylindrical (finitely additive) Gaussian measure on $K$. A bounded linear operator $U: K\to X$ is said to be  $\gamma$-radonifying iff  $U(\gamma_K)$ is a Borel Gaussian measure on $X$. 
\end{defn}

 Assume that $G:H\to H$ is $\gamma$-radonifying. Then the process $GL$ is a well defined L\'evy process  taking values in $H$. Under these assumptions the process $z$ defined by \eqref{eq_lin2} is a well defined $H$-valued process and moreover, it can be considered as a solution to the following integral equation:
\begin{equation}\label{eq_lin3}
z(t)=-\int_0^te^{-(t-s)A}\mathbf Cz(s)\,ds+\int_0^te^{-(t-s)A} G\,dL(s).
\end{equation}
With some abuse of notation, we will denote now by $\lambda_l$ the eigenvalues of the Stokes operator $A$, taking into account their mulitplicities, that is $\lambda_1\le\lambda_2\le\cdots$; and by $e_l$, the corresponding eigenvectors that form an orthonormal basis in $H$. We will impose a stronger condition on the operator $G$: 
\[Ge_l=\sigma_le_l,\quad l=1,2,\ldots\,.\]
We will consider the process 
\begin{equation*}
	z^0_t=\int^t_0 e^{-(t-s)A}GdL(s)=\sum^{\infty}_{l=1}z^0_l(t)e_l,
\end{equation*}
where 
\begin{align}\label{ztl}
	z^0_l(t)=\int^t_0 e^{-\lambda_l(t-s)}\sigma_l dL_l(s).
\end{align} 

\begin{lem}\label{asymlevy}
	Suppose that there exists some $\delta>0$ such that $\sum_{l\ge 1}|\sigma_l|^{\beta}\lambda^{\beta\delta}_l<\infty$. Then for all $p\in (0,\beta)$,
	\begin{align}\label{hypo}
		\E|A^{\delta}L(t)|^p\le C(\beta,p)\left(\sum_{l\ge 1}|\sigma_l|^{\beta}\lambda^{\beta\delta}_l\right)^{\frac{p}{\beta}} t^{\frac{p}{\beta}}<\infty.
	\end{align}
\end{lem}
\begin{proof}

Let $L(t)=\sum_{l\ge 1}L^l_t e_l$, $t\ge 0$ be the cylindrical $\beta$-stable process on $H$, where $e_l$ is the complete orthonormal system of eigenfunctions on $H$; and $L_1, L_2, \cdots, L_l$ are i.i.d. $\R$-valued, symmetric $\beta$-stable process on a common probability space $(\Omega,\mathcal{F},\P).$ Now take a bounded sequence of real number $\sigma=(\sigma_l)_{l\in\N}$. Let us define
\begin{align*}
	G_{\sigma} :H\to H;\quad G_{\sigma}u:=\sum^{\infty}_{l=1}\sigma_l\langle u, e_l\rangle e_l,
\end{align*}
and $\sigma_l$ are chosen such that
\begin{align*}
	G_{\sigma}L(t)=\sum^{\infty}_{l=1}\sigma_l \langle L_l(t), e_l\rangle e_l=\sum^{\infty}_{l=1}\sigma_l L_l(t)e_l.
\end{align*}
To show (\ref{hypo}), we follow the argument in the proof of Lemma 3.1 in \cite{MR3084156} and Theorem 4.4 in \cite{MR2773026}. Take a Rademacher sequence $\{r_l\}_{l\ge 1}$ in a new probability space $(\Omega',\mathcal{F}',\P')$, that is, $\{r_l\}_{l\ge 1}$ are i.i.d. with $\P\{r_l=1\}=\P\{r_l=-1\}=\frac{1}{2}.$ By the following Khintchine inequality: for any $p>0$, there exists some $C(p)>0$ such that for an arbitrary real sequence $\{h_l\}_{l\ge 1}$,
\begin{align*}
	\left(\sum_{l\ge 1}h^2_l\right)^{1/2}\le C(p)\left(\E'|\sum_{l\ge 1}r_l h_l|^p\right)^{1/p}.
\end{align*}
Via this inequality, we get
\begin{align*}
	\E| A^{\delta}L(t)|^q&=\E\left(\sum_{l\ge 1}\lambda^{2\delta}_l|\sigma_l|^2|L_l(t)|^2\right)^{p/2}\\
	&\le C\E\E'\left|\sum_{l\ge 1}r_l\lambda^{\delta}_l|\sigma_l||L_l(t)|\right|^p\\
	&=C\E'\E\left|\sum_{l\ge 1}r_l\lambda^{\delta}_l|\sigma_l||L_l(t)|\right|^p,
\end{align*}
where $C=C^p(p)$. For any $\lambda\in\R$, by the fact of $|r_k|=1$ and formula (4.7) of \cite{MR2773026},
\begin{align*}
	\E\exp\left\{i\eta\sum_{l\ge 1}r_l\eta^{\delta}_l|\sigma_l|L_l(t)\right\}=\exp\left\{-|\eta|^{\delta}\sum_{l\ge 1}|\sigma_l|^{\beta}\lambda^{\beta\delta}_l t\right\}.
\end{align*}
Now we know that  any symmetric $\beta$-stable r.v. $X\sim \tilde{S}_{\alpha}(\sigma,0,0)$ satisfies
\begin{align*}
	\E e^{i\eta       X}=e^{-\sigma^{\beta}\eta^{\beta}}
\end{align*}
for some $\beta\in (0,2)$, $\eta\in\R$. Then, for any $p\in (0,\beta)$,
\begin{align*}
	\E|X|^p=C(\beta,p)\sigma^p.
\end{align*}
Since $\sum_{l\ge 1}|\sigma_l|^{\beta}\lambda^{\beta\delta}_l<\infty$ , (\ref{hypo}) holds.
\end{proof}
\begin{lem}\label{lem_n1}
Suppose that there exists $\delta>0$ such that 
\[\sum_{l=1}^\infty|\sigma_l|^{\beta}\lambda^{\beta\delta}_l<\infty\,.\]
 Then for all $p\in (0,\beta)$ and $T>0$ 
	\begin{align}\label{Agam}
\E\sup_{0\le t\le T}|\hat{A}^{\delta}z_t|^p\le C\left(1+T^{p(1-\delta)}\right)T^{p/\beta}.		
	\end{align}
\end{lem}
\begin{proof}
It is proved in \cite{MR3084156} that for $p>1$ 
\begin{equation}\label{z0}
\E\sup_{0\le t\le T}|A^{\delta}z_t|^p\le C T^{p/\beta}\,.
\end{equation}
 In order to prove the lemma for the process $z$, we use formula \eqref{eq_lin3}. Let $Z=z-z^0$. Then \eqref{eq_lin3} yields
\[\frac{dZ}{dt}=-AZ-C\left( Z+z^0\right)=-\widehat AZ-Cz^0,\quad Z(0)=0\,.\]
Therefore, 
\[Z(t)=-\int_0^Te^{-(t-s)\widehat A}Cz^0(s)\,ds,\quad t\ge 0\,.\]
Then, by the properties of analytic semigroups we find that 
\[\begin{aligned}
\left|\widehat A^\delta Z(t)\right|&\le \int_0^t\left|\widehat A^\delta e^{-(t-s)\widehat A}\right|\left|Cz^0(s)\right|\,ds\\
&\le \sup_{s\le t}\left|Cz^0(s)\right|\int_0^t\frac{c}{(t-s)^\delta}\,ds\\
&\le c_1t^{1-\delta}\sup_{s\le t}\left|Cz^0(s)\right|\\
&\le c_1|C|t^{1-\delta}\sup_{s\le t}\left|z^0(s)\right|.
\end{aligned}\]
Since $\mathbf C$ is bounded, we have $D\left(\widehat A\right)=D(A)$ by Theorem 2.11 in \cite{MR710486}. Since $A\ge 0$ is selfadjoint, the domains of fractional powers can be identified as the complex interpolation spaces, see Section 1.15.3 of \cite{triebel}. Therefore, $D\left(A^\delta\right)=D\left(\widehat A^\delta\right)$ for every $\gamma\in(0,1)$, which yields the existence of constants, $r_1,r_2$ depending on $\delta$ only, such that
\[r_1\left|\widehat A^\delta x\right|\le\left|A^\delta x\right|\le r_2\left|\widehat A^\gamma x\right|,\quad x\in D\left(A^\gamma\right)\,.\]
Using \eqref{z0} we find that 
\[\mathbb E\sup_{t\le T}\left|A^\delta Z(t)\right|^p\le c_1^pr_2^p|C|^pT^{p(1-\delta)}\mathbb E\sup_{s\le T}\left|z^0(s)\right|^p<\infty.\]
Now the lemma follows since $z(t)=Z(t)+z^0(t)$. \\
Finally, for completeness we prove the case $p\in (0,1)$ for the process $z^0$. 
As (\ref{Agam}) is proved for $q\in (1,\beta)$, we fix $q\in (1,\beta)$ and then 
\begin{align*}
&\	\E\left(\sup_{0\le t\le T}|A^{\delta}z^0_t|^q\right)\le CT^{q/\beta}.
\end{align*}
Using the H\"older inequality (see for instance \cite{MR0367121}, p.191) one has
\begin{align*}
	\E(|X|^p\cdot 1)\le (\E X^{pq})^{1/q}.
\end{align*} 
We then have
\begin{align*}
&\	\E\left(\sup_{0\le t\le T}|A^{\delta}z^0_t|^p\right)\\
= &\ \E\left(\left\{\sup_{0\le t\le T}|A^{\delta}z^0_t|\right\}^{p}\right)\\
	\le &\ \E\left(\left\{\sup_{0\le t\le T}|A^{\delta}z^0_t|\right\}^{pq}\right)^{1/q}\\
	\le &\ \E\left(\left\{\sup_{0\le t\le T}|A^{\delta}z^0_t|\right\}^{q}\right)^{p/q}\\
	\le &\ (C_1 T^{q/\beta})^{p/q}\\
	= &\ C_1^{p/q}T^{p/\beta}\\
	\le &\ CT^{p/\beta}.
\end{align*}
\end{proof}
\begin{prop}\label{zlp}[p110,\cite{MR2773026}]
Suppose $\sum_{l\ge 1}\frac{\sigma^{\beta}_l}{\lambda+\alpha}<\infty$, then for any $0<p<\beta,$ $t\ge 0$,
\begin{align*}
	E|z^0_t|^p\le \tilde{c}_p\left(\sum^{\infty}_{l=1}|\sigma_l|^{\beta}\frac{1-e^{-\beta(\lambda_l+\alpha) t}}{\beta (\lambda_l+\alpha)}\right)^{p/\beta},
\end{align*}
where $c_{p}$ depends on $p$ and $\beta.$
Moreover, as $\alpha\to\infty$,
\begin{align*}
	\E|z^0_t|^p\to 0.
\end{align*}
\end{prop}
\begin{proof} 
In the spirit of the proof of Lemma \ref{asymlevy}, we follow the argument in the proof of Theorem 4.4 in \cite{MR2773026}. Let $z_t^0$ be the solution of
	\begin{align*}
		dz^0_t+(A+\alpha I)z^0_t=GdL(t),\quad z^0(0)=0
	\end{align*}
which has the expression
\begin{align*}
	z_t^0&=\int^t_0 S(t-s)G dL(s)\\
	&=\sum^{\infty}_{l=1}\left(\int^t_0 e^{-(\lambda_l+\alpha)(t-s)}\sigma_l dL^l_s\right)e_l,
\end{align*}
where we used the notation $S(t)=e^{-t(A+\alpha I)}$. Take a Radamacher sequence $\{r_l\}_{l\ge 1}$ in a new probability space $(\Omega',\mathcal{F}',\P')$, that is $\{r_l\}_{l\ge 1}$ are i.i.d. with $\P(r_l=1)=\P(r_l=-1)=\frac{1}{2}.$ By the following Khintchine inequality: for any $p>0$, there exists some $c_p>0$ such that for any arbitrary real sequence $\{c_l\}_{l\in\N}$,
\begin{align*}
	\left(\sum_{l\ge 1}c^2_l\right)^{1/2}\le c_p\left(\E'|\sum_{l\ge 1}r_l c_l|^p\right)^{1/p},
\end{align*}
where $c_p$ depends only on $p$.

Now fixing $\omega\in\Omega$, $t\ge 0$, we write
\begin{align*}
	\left(\sum_{l\ge 1}|z^0_l(t,\omega)|^2\right)^{1/2}\le c_p(\E'|\sum_{l\ge 1}r_lz^0_l(t,\omega)|^p)^{1/p}.
\end{align*}
Then
\begin{align*}
	\E|z^0_t|^p&=\left(\sum^{\infty}_{l=1}|\int^t_0 e^{-(\lambda_l+\alpha)(t-s)}\sigma_ldL^l_s|^2\right)^{\frac{p}{2}}\\
	&\le c^p_p\E\left(\E'|\sum^{\infty}_{l=1}r_l z^0_l(t)|^p\right)=c^p_p\E'\left(\E|\sum_{l=1}r_l z^l_t|^p\right)=c^p_p\E'\left(\E|\sum^{\infty}_{l=1}r_l \int^t_0 e^{-(\lambda_l+\alpha)(t-s)}\sigma_ldL^l_s|^p\right).
\end{align*}
For any $t\ge 0$, $\kappa\in\R$ using the fact $|r_l|=1$, $l\ge 1$ and formula (4.7) in \cite{MR2773026},
\begin{align*}
	\E e^{i\kappa\sum_{l\ge 1}r_l z^0_l(t)}=e^{-|\kappa|^{\beta}}\sum_{l\ge 1}|\sigma_l|^{\beta}\int^t_0 e^{-\beta(\lambda_l+\alpha)(t-s)}ds.
\end{align*}
Now we use (3.2) in \cite{MR2773026}: If $X$ is a symmetric $\beta$-stable r.v. with distribution $S(\beta,\gamma,0)$ satisfying
\begin{align*}
	\E e^{i\kappa X}=e^{-\gamma^{\beta}|\kappa|^{\beta}}
\end{align*}
for some $\beta\in (0,2)$ and any $\kappa\in\R$, then for any $p\in(0,\beta)$, one has
\begin{align*}
	\E X^p=C(\beta,p)\gamma^p.
\end{align*}
Since $\sum_{l\ge 1}\frac{\sigma^{\beta}_l}{\lambda_l+\alpha}<\infty$, the assertion follows. Furthermore, $\E|z_t|^0_p\to 0$ as $\alpha\to\infty.$
\end{proof}

Let us now recall the definition of Skorohod space $D=D(a,b; E)$, which consists of a function $x:[0,T]\to E$ which admits a limit $x(t-)$ from the left at each point $t\in (0,T]$ and the limit $x(t+)$ from the right at each point $t\in (0,T].$ The 
Skorohod space can be endowed with a metric topology such that it becomes a complete separable metric space(See for instance Billingsley \cite{MR2893652}).

Here we present a Lemma that allows us to claim that the solution of SNSEs has c\`adl\`ag trajectories. The proof follows closely with Lemma 3.3 in \cite{MR3084156}.
	\begin{lem}\label{cadlaglem}
	Assume that for a certain $\delta\in[0,1)$ 
	\[\sum_{l=1}^\infty|\sigma_l|^{\beta}\lambda^{\beta\delta}_l<\infty\,.\]
	
		Then the process $z$ defined by (\ref{ztl}) has a version in $D\left([0,\infty];D\left(A^\delta\right)\right).$
	\end{lem}
	\begin{proof}
	By Lemma \ref{lem_n1} we have 
	\[\E\sup_{0\le t\le T}|A^{\delta}z_t|^p<\infty\]
for any $p\in(0,\beta)$. Now, by Theorem 2.2 in  \cite{liu2012note}  $z^0$ has a c\`adl\`ag modification\footnote{Modification with a c\`adl\`ag path.} in $V$. By representation \eqref{eq_lin3} the process $z$ is c\`adl\`ag as well, and the proof of the Lemma is completed.
	\end{proof}    
\par\bigskip\noindent
Let $B: H\to H$ be a selfadjoint operator with the complete orthonormal system of eigenfunctions $(e_l)\subset L^p(\sph^2)$ and the corresponding set of eigenvalues $(\lambda_l)$. It follows from Theorem 2.3 of \cite{MR2051026} that if further $B$ has a compact inverse $B^{-1}$, then the operator $U^{-s}: H\to L^p(\sph^2)$ is well-defined and $\gamma$-radonifying iff
\begin{align}\label{Uminuss}
	\int_{\sph^2}\left(\sum_l \lambda^{-2s}_l|e_l(x)|^2\right)^{p/2}dS(x)<\infty.
\end{align}

In what follows we will study the $\gamma$- radonifying property.
 
\begin{lem}\label{radonifys}
	Let $\vec{\Delta}$ denotes the Laplace-de Rham operator on $\sph^2$ and $q\in (1,\infty).$ Then the operator
	\begin{align*}
		(-\vec{\Delta}+1)^{-s}: H\to L^q(\sph^2)\,\,\text{is}\,\,\gamma-\text{radonifying iff }s>1/2.
	\end{align*}
\end{lem}
\begin{proof}
See proof of Lemma 3.1 in \cite{BGQ17}.
\end{proof}
Let $X=\mathbb{L}^4(\sph^2)\cap H$ be the Banach space endowed with the norm
\begin{align*}
	|x|_X=|x|_H+|x|_{\mathbb{L}^4(\sph^2)}.
\end{align*}
It follows from Lemma \ref{radonifys} that the operator
\begin{align}\label{radons}
	A^{-s}: H\to X\,\,\text{is }\gamma-\text{radonifying iff }s>1/2.
\end{align}

One has to choose $X$ wisely, so that $U: K\to X$ is $\gamma$-radonifying in checking validifty of subordinator condition as in p.156, \cite{MR2584982}.      
The following is our standing assumption.
\\
    
\textbf{Assumption 1.} A continuously embedded Hilbert space $K\subset H\cap \mathbb{L}^4$ is such that for any $\delta\in (0,1/2)$,
\begin{align}
	A^{-\delta}: K\to H\cap \mathbb{L}^4\quad\text{is $\gamma$-radonifying. }
\end{align}
It follows from (\ref{radons}) that if $K=D(A^s)$ for some $s>0$, then assumption 1 is satisfied.
\begin{rmk}
	Under the above assumption, we have the fact that $K\subset H$ and Banach space $X$ is taken as $H\cap L^4$. In fact, space $K:=Q^{1/2}(W)$ is the RKHS of noise $W(t)$ on $H\cap \mathbb{L}^4$ with the inner product $\langle\cdot,\cdot\rangle_K=\langle Q^{-1/2}x,Q^{-1/2}y\rangle_W$, $x,y\in K$. The notation $Q$ denotes the covariance of the noise $W$.
\end{rmk}

Note: The parameters used in Lemma \ref{radonifys} and Assumption 1 are independent. In Lemma \ref{radonifys}, we start with the whole space, a smaller exponent is required to map onto $H\cap\mathbb{L}^4(\mathbb{S}^2)$, so the assumption $s>1/2$ is justified. On the other hand, in Assumption 1, we start with a smaller space, so a bigger exponent is required to map onto $H\cap\mathbb{L}^4(\mathbb{S}^2)$, so $\delta\in (0,1/2).$
\begin{cor}In the framework of Proposition \ref{analytics}, let us additionally assume that there exists a separable Hilbert space $K\subset X$ such that the operator $A^{-\delta}: K\to X$ is $\gamma$-radonifying for some $\delta\in(0,\frac12)$. Then
	\begin{align*}
		\int^{\infty}_0|e^{-tA}|^2_{R(K,X)}dt<\infty.
	\end{align*}
\end{cor}
\begin{proof}
	Since $e^{-tA}=A^{\delta}e^{-tA}A^{-\delta}$, it follows by Neidhardt \cite{neidhardt1978stochastic} that
\begin{align*}
	|e^{-tA}|_{R(K,X)}\le |A^{\delta}e^{-sA}|_{\mathcal{L}(X,X)}|A^{-\delta}|_{R(K,X)},
\end{align*}
and then Proposition \ref{analytics} yields finiteness of the integral. 
\end{proof}
Let us recall what one means by $M$-type $p$ Banach space (see for instance \cite{MR1313905}). Suppose $p\in[1,2]$ is fixed, then the Banach space $E$ is called type $p$, iff there exists a constant $K_p(E)>0$, such that 
\begin{align*}
	\E\left|\sum^n_{i=1}\xi_i x_i\right|^p\le K_p(E)\sum^n_{i=1}|x_i|^p.
\end{align*}
for any finite sequence of symmetric i.i.d. random variables $\xi_1,\cdots, \xi_n: \Omega\to [-1,1],n\in\N$, and any finite sequence $x_1,\cdots, x_n$ from $E$.

Moreover, a Banach space $E$ is of martingale type $p$ iff there exists $L_p(E)>0$ such that for any $E$-valued  martingale $\{M_n\}^N_{n=0}$ the following holds:
\begin{align*}
	\sup_{n\le N}\E|M_n|^p\le L_p(E)\sum^N_{n=0}\E|M_n-M_{n-1}|^p.
\end{align*}
The following is an abstract result from \cite{huang2013random} which will be needed for the rest of this paper.
\begin{lem}[Corollary 8.1,\cite{huang2013random}]\label{leml4}
Assume that: $p\in (1,2]$; $X$ is a subordinator L\'evy process from the class $\text{Sub}(p)$; $E$ is a separable type $p$ Banach space; $U$ is a separable Hilbert space; $E\subset U$; and $W=(W(t),t\ge 0)$ is an $U$-valued Wiener process.

Define a $U$-valued L\'evy process as
\begin{align*}
	L(t)=W(X(t)),\quad t\ge 0\,.
\end{align*}	
Then the $E$-valued process 
\begin{align*}
	z(t)=\int^t_0 e^{-(t-s)(\vec{A}+\alpha I)}dL(s)
\end{align*}
is well defined. Moreover, with probability 1, for all $T>0$,
\begin{align*}
	\int^T_0 |z(t)|^p_E dt<\infty,
\end{align*}

\begin{align*}
	\int^T_0 |z(t)|^4_{L^4} dt<\infty.
\end{align*}
\end{lem}
The following existence and regularity result is a version of the result in \cite{MR2584982}.
\begin{thm}
Let the process $L$ be defined in the same way as in Lemma \ref{leml4}. Assume that one of the following conditions is satisfied:
	\begin{itemize}
		\item [(i)] $p\in (0,1]$ or
		\item [(ii)] the Banach space $E$ is separable and of martingale type $p$ for a certain $p\in(1,2]$. \end{itemize}
		Then the process 
	\begin{align}\label{stokes}
		z_{\alpha}(t)=\int^t_{-\infty}e^{-(t-s)(\hat{A}+\alpha I)}dL(s)
	\end{align} 
is well defined in $E$ for all $t>0$. Moreover, if $p\in (1,2]$,  then the process $z$ of (\ref{stokes}) is  c\`adl\`ag. 
\end{thm}
\begin{proof}
	As $S=(S(t),t\ge 0)$ is a $C_0$ semigroup in the separable martingale type $p$-Banach space $E$, there exists a Hilbert space $H$ as the reproducing Kernel Hilbert space of $W(1)$ such that the embedding $i: H\hookrightarrow E$ is $\gamma$-radonifying. The proof of this theorem is a straight forward application of Theorem 4.1 and 4.4 in \cite{MR2584982}.
\end{proof}
In order to obtain well-posedness of (\ref{SNSE4}), one needs some regularity on the noise term. Fortunately, this becomes attainable using Lemma \ref{leml4} . In view of this, we construct the driving L\'evy noise $L=L(t)$ by subordinating a cylindrical Wiener process $W$ on a Hilbert space $H$ as defined in (\ref{abstractH}).    
Let $\{W^l_t,t\ge 0\}$ be a sequence of independent standard one-dimensional Wiener processes on some given probability space $(\Omega,\mathcal{F},\P)$. The cylindrical Wiener process on $H$ is defined by
\begin{align*}
	W(t):=\sum_l W^l_t e_l,
\end{align*}
where $e_l$ is the complete orthonormal system of eigenfunctions on $H$.

For $\beta\in (0,2)$, let $X(t)$ be an independent symmetric $\beta/2$-stable subordinator. That is, an increasing, one dimensional L\'evy process with the Laplace Transform
\begin{align*}
	\E e^{-r X(t)}=e^{-t|r|^{\beta/2}},\quad r>0.
\end{align*}
The subordinated cylindrical Wiener process $\{L(t),t\ge 0\}$ on $H$ is defined by
\begin{align*}
	L(t):=W(X(t)),\quad t\ge 0.
\end{align*}
Note in general that $L(t)$ does not belong to $H.$ More precisly, $L(t)$ lives on some larger Hilbert space $U$ with the $\gamma$-radonifying embedding $H\hookrightarrow U$.
Now, let us consider the
abstract It\^o equation in (\ref{asnse4}) (which we restate here) in $H=L^2(\mathbb{S}^2):$
\begin{align}\label{snse}
	d u(t)+\nu A u(t)dt+B( u(t), u(t))dt+ \mathbf{C} u= fdt+GdL(t),\quad  u(0)= u_0.
\end{align}
Writing (\ref{asnse4}) into the usual mild form one has
\begin{align}\label{mild}
	 u(t)=S(t) u_0-\int^t_0 S(t-s) B( u(s))ds+\int^t_0 S(t-s)fds+\int^t_0 S(t-s)GdL(s).
\end{align}
where $S(t)$ is an analytic $C_0$ semigroup $(e^{-t\hat{A}})$ generated by $\hat{A}=\nu A+ \mathbf{C}$, and $A$ is the Stokes operator in $H$. Note that $\hat{A}$ is a strictly positive selfadjoint operator in $H$, that is $A: D(A)\subset H\to H$, $\hat{A}=\hat{A}^*>0$, $\langle Av,v\rangle\ge\gamma|v|^2$ for any $v\in D(A)$ for some $\gamma>0$ and $v\ne 0$. The operator $G:H\to H$ is a bounded linear operator. For a fixed $\alpha>0$ we introduce the process 
\[z_{\alpha}(t) :=\int_0^t e^{-(t-s)(\alpha+\widehat A)}GdL(t)\]
that solves the OU equation
\begin{align}\label{Ornstein Uhlenbeck}
	d z_{\alpha}+(\nu  A + \mathbf{C}+\alpha) z_{\alpha}dt=GdL(t),\quad t\ge 0\,.
\end{align}
Now let $ v(t)= u(t)-z_\alpha(t)$.  Then
\begin{align*}
	\begin{cases}
		d v(t)+\nu A( u(t)- z_{\alpha}(t))dt+B( u(t))dt+ \mathbf{C}( u- z_{\alpha}(t))dt-\alpha  z_{\alpha}(t)dt= fdt,\\
	 v(0)= v_0.
	\end{cases}
\end{align*}
The problem becomes
 \begin{align*}
 	\begin{cases} 		
		d v(t)+\nu  A  v(t)dt+B( v(t)+ z_{\alpha}(t))dt+ \mathbf{C} v(t)dt-\alpha  z_{\alpha}(t)dt= f dt, \\
 		 v(0)= v_0.
 	\end{cases}
 \end{align*}
 Converting into the standard form,
 \begin{align}\label{ODE}
 	\begin{cases} 		
		\frac{d^+}{dt} v(t)+(\nu A+ \mathbf{C}) v(t)= f-B( v(t)+ z_{\alpha}(t))+\alpha  z_{\alpha}(t),\\
 		 v(0)= v_0.
 	\end{cases}
 \end{align}
 where $\frac{d^+v}{dt}$ is the right-hand derivative of $v(t)$ at $t$.
Solution to equation \eqref{ODE} will be understood in the mild sense, that is as a solution to the integral equation 
\begin{align}\label{mildv}
 v(t)=S(t) v(0)+\int^t_0 S(t-s)(f-B( v(s)+ z_{\alpha}(s))+\alpha  z_{\alpha}(s))ds,
\end{align}
with $ v_0= u_0- z_{\alpha}(0).$
One can easily show that (\ref{ODE}) and (\ref{mildv}) are equivalent 
 for $ v\in C(0,\infty;V)\cap L^2_{\text{loc}}(0,\infty;D(A)).$ More precisely, (\ref{mildv}) follows from (\ref{ODE}) via integration. Then (\ref{ODE}) follows from (\ref{mildv}) via the usual continuity argument  (see Lebesgue Dominated Convergence Theorem in the Appendix), namely, differentiating the integral when the integrand is continuous.

For brevity, we write $z_{\alpha}$ as $z$.
Let us now explain what is meant by a  solution of (\ref{asnse4}).

\begin{defn}
	Suppose that $z\in L^4_{\text{loc}}([0,T);\mathbb{L}^4(\mathbb{S}^2)\cap H)$, $ v_0\in H$, $f\in V'$. A weak solution to (\ref{asnse4}) is a function $v\in C([0,T);H)\cap L^2_{\text{loc}}([0,T);V)$ which satisfies (\ref{ODE}) in a weak sense for any $\phi\in V$, $T>0$, and
	\begin{align}\label{weakode}
\partial_t( v,\phi)=( v_0,\phi)-\nu( v,A\phi)-b( v+ z, v+ z,\phi)-( \mathbf{C} v,\phi)+(\alpha z+ f,\phi).
	\end{align}
Equivalently, (\ref{ODE}) holds as an equality in $V'$ for a.e. $t\in[0,T]$. 
\end{defn}
Now if $ f\in H$, and the following regularity is satisfied,
\begin{align}
	 v\in L^{\infty}(0,T;V)\cap L^2(0,T;D(A)),
\end{align}
then the solution becomes strong. More precisely,
\begin{defn}[Strong solution]\label{ssoln}
Suppose that $z\in L^4_{\text{loc}}([0,T);\mathbb{L}^4(\mathbb{S}^2)\cap H)$, $ v_0\in V$, $f\in H$. We say that $u$ is a \emph{strong solution} of the stochastic Navier-Stokes equations (\ref{asnse4}) on the time interval $[0,T]$ if $u$ is a weak solution of (\ref{asnse4}) and in addition
\begin{align}
	 u\in L^{\infty}(0,T;V)\cap L^2(0,T;D(A)).
\end{align}
\end{defn}

The main theorems proved in this paper are the following.

\begin{thm}\label{t3}
	Assume that $\alpha\ge 0$, $ z\in L^4_{\text{loc}}([0,\infty);\mathbb{L}^4(\mathbb{S}^2)\cap H)$, $ f\in H$ and $ v_0\in H.$ Then, there exists a unique solution of (\ref{mildv}) in the space $C(0,T;H)\cap L^2(0,T;V)$ which belongs to $C(h,T; V)\cap L^2_{\text{loc}}(h,T;D(A))$ for all $h>0$ and $T>0.$ Moreover, if $ v_0\in V$, then $ v\in C(0,T; V)\cap L^2_{\text{loc}}(0,T;D(A))$ for all $T>0.$ In particular, $ v(T,z_n) u^0_n\to  v(T,z_n) u_0$ in $H.$ Moreover, if
	\[\sum_{l=1}^\infty|\sigma_l|^{\beta}\lambda^{\beta/2}_l<\infty\,,\]
	then the theorem holds.
\end{thm}
\vspace{1cm}
\begin{thm}\label{t4}
	Assume that $\alpha\ge 0$, $z\in L^4_{\text{loc}}([0,\infty);\mathbb{L}^4(\mathbb{S}^2)\cap H)$, $ f\in H$ and $ v_0\in H.$ Then, there exists  $\P$-a.s. unique solution of (\ref{asnse4}) in the space $D(0,T;H)\cap L^2(0,T;V),$ which belongs to $D(\epsilon,T; V)\cap L^2_{\text{loc}}(\epsilon,T;D(A))$ for all $\epsilon>0,$ and $T>0.$ Moreover, if $ v_0\in V$, then $ u\in D(0,T; V)\cap L^2_{\text{loc}}(0,T;D(A))$ for all $T>0$, $\omega\in\Omega$. Moreover, if
	\[\sum_{l=1}^\infty|\sigma_l|^{\beta}\lambda^{\beta/2}_l<\infty\,,\]
	then the theorem holds.
\end{thm}

\section{Proof of Theorem \ref{t3}: Strong solutions}\label{sec:num45}
Suppose now $ f\in H$. In what proceeds we will show that if $ u_0\in V$ then we obtain a more regular kind of solution and deduce that if $ v_0\in H$ then $ v(t)\in V$ for every $t>0.$ In this paper, we will construct a unique global strong solution (in the sense of Definition \ref{ssoln}).
\\

The proof of Theorem \ref{t3} follows closely to Theorem 3.1 in \cite{MR1207308}. However in the proof in \cite{MR1207308} there is no Coriolis force and additive noise, whereas here there are. In particular, our constants in the proof now depend on $|F(t)|$ , $|z(t)|$ and $|z(t)|_V$, but not on the Coriolis term due to the antisymmetric condition $(\mathbf{C}v,Av)=0.$

\begin{rmk}
    One can alternatively prove Theorem \ref{t3} via the usual Galerkin approximation.
\end{rmk}     
\vspace{1cm}         
\subsection{Existence and uniqueness of a strong solution with $ v_0\in V$}
\vspace{-0.4cm}
The following function spaces are introduced for convenience.
\begin{defn}
The spaces
    \begin{align}
        X_T:=C(0,T;H)\cap L^2(0,T;V),
    \end{align}
    \begin{align}
            Y_T=C(0,T;V)\cap L^2(0,T;D(A))
    \end{align}
are endowed with the norms
\begin{align*}
    |\cdot|_{X_T}:=|\cdot|_{C(0,T;H)}+|\cdot|_{L^2(0,T;V)},
\end{align*}
\begin{align*}
    |\cdot|_{Y_T}:=|\cdot|_{C(0,T;V)}+|\cdot|_{L^2(0,T;D(A))}.
\end{align*}
\end{defn}
Or explicitly,
    \begin{align*}
        |f|^2_{X_T}&=\sup_{0\le t\le T}|f(t)|^2+\int^T_0|f(s)|^2_Vds,
    \end{align*}
    \begin{align*}
        |f|^2_{Y_T}&=\sup_{0\le t\le T}|f(t)|^2_V+\int^T_0|Af(s)|^2ds.
    \end{align*}

Let $\mathcal{K}$ be the map in $Y_T$ defined by
\begin{align*}
    \mathcal{K}( u)(t)=\int^t_0 S(t-s)B( u(s))ds,\quad t\in[0,T],\,\,u\in Y_T.
\end{align*}
The following is a crucial lemma for the proof of existence and uniqueness.
\begin{lem}
    There exists $c>0$ such that for every $ u, v\in Y_T$,
    \begin{align*}
        |\mathcal{K}( u)|^2_{Y_T}\le c| u|^2_{Y_T}\sqrt{T},
    \end{align*}
    \begin{align*}
    |\mathcal{K}( u)-\mathcal{K}( v)|^2_{Y_T}\le c| u- v|^2_{Y_T}(| u|^2_{Y_T}+| v|^2_{Y_T})\sqrt{T}.   
    \end{align*}
\end{lem}
\begin{proof}
    Recall the classical facts due to Lions \cite{MR0291887},
    \begin{itemize}
        \item for any $ f\in L^2(0,T;H)$, the function $t\mapsto  x(t)=\int^t_0 S(t-s) f(s)ds$ belongs to $Y_T$ and
        \item the map $ f\mapsto  x$ is continuous from $L^2(0,T;H)$ to $Y_T$.
    \end{itemize}
We remark that the second fact implies that $\int^t_0| f(s)|^2_Hds<\infty$.
Now because $B( u)\in L^2(0,T;H)$, that is $\int^t_0|B( u(s))|^2_Hds$, using the previous classical facts, combining with (\ref{b4}) one has,
\begin{align*}
    |\mathcal{K}( u)|^2_{Y_T}&\le c_1\int^T_0|B( u(s))|^2_Hds\\
    &\le c_2\int^T_0 | u|^2_V| u|_V| A u| dt\\
    &\le c_2\sup_{0\le t\le T}| u(t)|^2_V\int^T_0| u(t)|_V|A u(t)|dt\\
    &\le \frac{c_2}{2}\sup_{0\le t\le T}| u(t)|^2_V\left(\int^T_0| u(t)|^2_V+| A u(t)|^2dt\right)\\
    &\le c_3| u|^4_{Y_T}\sqrt{T}.
\end{align*}
Similarly, combining Lions' results and $(\ref{b4})$, one has
\begin{align*}
    \left|\mathcal{K}( u)-\mathcal{K}( v)\right|^2_{Y_T}&\le c_4\int^T_0|B( u- v, u)+B( v, u- v)|^2_H dt\\
    &\le c_5\int^T_0|B( u- v, u)|^2_H+|B( v, u- v)|^2_Hdt\\
    &\le c_5\int^T_0 c_7| u- v|^2_V| u|_V| A u|+c_8| u- v|^2_V| v|_V|A v|dt\\
    &\le c| u- v|^2_{Y_T}(| u|^2_{Y_T}+| v|^2_{Y_T})\sqrt{T}.
\end{align*}
\end{proof}

\begin{lem}\label{lemfh}
    Assume that $\alpha\ge 0$, $z\in L^4_{\text{loc}}([0,\infty);\mathbb{L}^4(\mathbb{S}^2)\cap H)$, $ f\in H$ and $ v_0\in V.$ Then, there exists unique solution of (\ref{mild}) in the space $C(0,T;V)\cap L^2(0,T;D(A))$ for all $T>0.$
\end{lem}
\begin{proof}
First let us prove local existence and uniqueness.
     Let $Y_{\tau}=C(0,\tau;V)\cap L^2(0,\tau;D(A))$ be equipped with the norm
    \begin{align*}      |f|^2_{Y_{\tau}}=\sup_{t\le\tau}|f(t)|^2+\int_{0}^{\tau}|Af(s)|^2\,ds,
    \end{align*}
And let $\Gamma$ be a nonlinear mapping in $Y_{\tau}$ as
\begin{align*}
    (\Gamma  v)(t)=S(t) v_0+\int^t_0 S(t-s)(f-B( v(s)+  z(s))+\alpha   z(s))ds.
\end{align*}
    Now recall the following classical result due to Lion:
    \begin{align*}
        \begin{cases}
        A1\quad    S(\cdot) v_0\in Y_{\tau},\,\,\forall\,\, v_0\in H, \tau>0; \\
    A2\quad    \mbox{The map $t\mapsto x(t)=\int_0^{t}S(t-s) f(s)ds$ belongs to $Y_{\tau}$ for all $L^2(0,\tau;H)$} ;   \\
    A3\quad    \mbox{The mapping $f\mapsto x$ is continuous from $L^2(0,\tau;H)$ to $Y_{\tau}$}.
        \end{cases}
    \end{align*}
    Note, our assumption $  z(t)\in L^4([0,\infty);L^4(\mathbb{S}^2)\cap H)$ 
 implies that $z(t)\in Y_{\tau}$ as $  z(t)$ is square integrable and $V$ can be continuously embedded into $L^4(\mathbb{S}^2)$ .

    The first step is to show $\Gamma$ is well defined.
     Using  assumptions A1 and A2 and the assumption for $  z(t)$, together with Young inequality, one can show that
    \begin{align*}
        |\Gamma|^2_{Y_{\tau}}
&\le c\left|S(t) v_0\right|^2_{Y_{\tau}}
+c\left|\int^t_0S(t-s)B( v(s)+  z(s))ds\right|^2_{Y_{\tau}}+c\left|\int^t_0S(t-s) fds\right|^2_{Y_{\tau}}+c\alpha\left|\int^t_0 S(t-s)   z(s)\right|^2_{Y_{\tau}},
    \end{align*}
for some different constant $c$.
Now due to $A_1$ and $A_2$, the first and third terms are finite, due to $A_2$ and the trilinear inequality (\ref{b5}). The second term is finite. The last term also finite due to the assumption on $z(t)$ that 
\begin{align}
        |\Gamma|^2_{Y_{\tau}}&\le c_1+c_2| v|^4_{Y_{\tau}}+c_3+c_4.
\end{align}
Whence the map $\Gamma v$ is well defined in $Y_{\tau}$, and $\Gamma$ maps $Y_{\tau}$ into itself.

Now we have
    \begin{align*}
    &\    |\Gamma( v_1)-\Gamma( v_2)|^2_{Y_{\tau}}\\
    &\ \le |\int^{\tau}_0 S(t-s)(B( v_1(s)+  z(s))-B( v_2(s)+  z))ds|^2_{Y_{\tau}}\\
&\ \le c_6| v_1- v_2|^2_{Y_{\tau}}(| v_1+ z|^2_{Y_{\tau}}+| v_2+ z|^2_{Y_{\tau}})\sqrt{\tau},
    \end{align*}
for all $ v_1$, $ v_2$ and $  z$ in $Y_{\tau}.$ Therefore, for sufficiently small $\tau>0$, $\Gamma$ is a contraction in a closed ball of $Y_{\tau}$, yielding existence and uniqueness of a local solution of (\ref{mildv}) in $Y_{\tau}.$ That is, the solutions are bounded in $V$ on some short time interval $[0,\tau).$
\begin{rmk}
If the following map
        \begin{align*}
            (\Gamma  u)(t)=S(t) u_0-\int^t_0 S(t-s)B( u(s))ds+\int^t_0 S(t-s) fds+\int^t_0 S(t-s)GdL(s)
        \end{align*}
is used to prove contraction, then one would have to assume
    \begin{align*}
        \int^T_0|Az(t)|^2dt<\infty.
    \end{align*}    
\end{rmk}
\vspace{1cm}
The local existence and uniqueness results indicate that the solution can be extended up to the maximal lifetime $T_{f,z}$ and then is well defined on the right-open interval $[0,T_{f,z}).$ 
Next, we will prove the local solution may be continued to the global solution which is valid for all $t>0$, in the class of weak solutions satisfying a certain energy inequality. This is consistent with the results for the 2D NSEs that, a strong solution exists globally in time and is unique. See for instance Theorem 7.4 of Foias and Temam \cite{MR1855030}.
          
It suffices to find an uniform apriori estimate for the solution $v$ in the space $Y_{T_0}$ such that for any $T_0\in[0,T_{f,z})$:
    \begin{align}\label{YT}
        | v|^2_{Y_{T_0}}\le C\quad\text{for all}\quad T_0\in [0,T_{f, z}),
    \end{align}
where $C$ is independent of $T_0.$
This uniform apriori estimate along with the local existence-uniqueness proved earlier, yields the unique global solution $u$ in $Y_{T,z}$. Moreover, this solution exists globally in time. 
Hence one can deduce that the solution is well defined up to the time $t=T_{f,z}$. At this point in time the iterated process could be repeated and the solution can be found on $[T_{f,z},2T_{f,z}]$ and so forth. Hence the solution can be found in $C(0,\infty;V)\cap L^2_{\text{loc}}(0,\infty;D(A)).$
To prove (\ref{YT}), we first need to show
\begin{align*}
    | v|_{X_{T_0}}\le c_0.
\end{align*}
Toward that end, we work with a modified version of (\ref{ODE})
\begin{align}\label{mode}
    \begin{cases}
        \partial_t  v+\nu A v=-B( v)-B( v, z)-B( z, v)-\mathbf{C} v+F, \\
     v(0)= v_0  .  
    \end{cases}
\end{align}
where $F=-B( z)+\alpha  z+ f$ is an element of $H$, since the $H$ norm of all of its three terms is bounded. Now multiplying both sides with $ v$, and integrating over $\mathbb{S}^2$, one obtains
\begin{align*}
    \partial_t| v|^2+\nu | v|^2_V&=-b( v, v, v)-b( v, z, v)-b( z, v, v)-(\mathbf{C} v, v)+\langle F, v\rangle\\
    &=b( v, v, z)+(F, v).
\end{align*}
Now by (\ref{b1}), one has
\begin{align*}
    |b( v, v, z)|&\le c| v|| v|_V| z|.\\
\intertext{Then applying Young inequality with $ab=\sqrt{\frac{\epsilon}{2}}| v|_V|v|\sqrt{\frac{2}{\epsilon}}| z|_V$ it follows that}
&\le \frac{\epsilon| v|^2_V}{4}+\frac{1}{\epsilon}| v|^2| z|^2_V.
\end{align*}
On the other hand,
\begin{align*}
    (F(t), v)=|F(t)|| v|\le\frac{1}{\epsilon}|F(t)|^2+\frac{\epsilon}{4}| v|^2.
\end{align*}
So that
\begin{align}\label{ptv2}
    \partial_t| v|^2+(2\nu-\frac{\epsilon}{2}) | v|^2_V&\le \frac{2}{\epsilon}| v|^2| z|^2_V+\frac{2}{\epsilon}|F(t)|^2+\frac{\epsilon}{2}| v|^2
\end{align}
for all $\epsilon>0$.

By integrating in $t$ from $0$ to $T$, after simplifying, one obtains
\begin{align}\label{integratedvh1norm}
    \int^T_0| v(t)|^2_V\le\frac{1}{2\nu-\frac{\epsilon}{2}}\left(| v(0)|^2+\frac{2}{\epsilon}\int^T_0| v(t)|^2| z(t)|^2_V dt+\frac{2}{\epsilon}\int^T_0|F(t)|^2dt+\frac{\epsilon}{2}\int^T_0| v(t)|^2 dt\right)\le K_1,
\end{align}
Since $ v(0)= u_0$, 
\begin{align*}
    K_1=K_1( u_0,F,\nu,T, z).
\end{align*}
On the other hand, by integrating (\ref{ptv2}) in $t$
from $0$ to $s$,  $0<s<T$, we obtain
\begin{align*}
    | v(s)|^2\le | u_0|^2+\frac{2}{\epsilon}\int^s_0| v(t)|^2| z(t)|^2_V dt+\frac{2}{\epsilon}\int^s_0|F(t)|^2dt+\frac{\epsilon}{2}\int^s_0|v(t)|^2dt,
\end{align*}

\begin{align*}   \sup_{s\in[0,T_{f,z}]}|v(s)|^2\le K_2,
\end{align*}
\begin{align*}
    K_2=K_2( u_0, F,\nu,T, z)=(2\nu-\frac{\epsilon}{2})K_1.
\end{align*}
 Hence, for any $\epsilon$ such that $\frac{\epsilon}{2}<2\nu,$ applying the Gronwall lemma to
\begin{align*}
    \partial_t| v|^2\le \left(\frac{2}{\epsilon}| z|^2_V+\frac{\epsilon}{2}\right)| v|^2+\frac{2}{\epsilon}|F(t)|^2,
\end{align*} one obtains
\begin{align*}   |v(t)|^2&\le|v(0)|^2\exp\left(\int^{t}_0\frac{2}{\epsilon}| z(\tau)|^2_V+\frac{\epsilon}{2}d\tau\right)| v|^2+\int_0^t\frac{2}{\epsilon}|F(s)|^2\exp(\int^t_s\left(\frac{2}{\epsilon}| z(\tau)|^2_V+\frac{\epsilon}{2}\right)d\tau)ds.
\end{align*}
And so
\begin{align*}
    \sup_{t\in [0,T_{f,z}]}| v(t)|^2\le | v(0)|^2\exp\left(\int^{T_{f,z}}_0\frac{2}{\epsilon}| z(\tau)|^2_V+\frac{\epsilon}{2}d\tau\right)+\int_0^{T_{f,z}}\frac{2}{\epsilon}|F(s)|^2\exp(\int^{T_{f,z}}_s\left(\frac{2}{\epsilon}| z(\tau)|^2_V+\frac{\epsilon}{2}\right)d\tau)ds.
\end{align*}
To avoid clumsiness, we write momentarily $T_{f,z}=T$. Let
\begin{align}
    \psi_{T}( z)=\exp\left(\int^T_0\frac{2}{\epsilon}| z(\tau)|^2_V+\frac{\epsilon}{2}d\tau\right)<\infty,\quad c_F=\int^{T}_0\frac{2}{\epsilon}|F(s)|^2\exp\left(\int^{T}_s\left(\frac{2}{\epsilon}| z(\tau)|^2_V+\frac{\epsilon}{2}\right)d\tau\right)ds.
\end{align}
So
\begin{align}
    \sup_{t\in[0,T]}| v(t)|^2\le | v(0)|^2\psi_T( z)+c_F,
\end{align}
which implies
\begin{align}\label{linfx}
     v\in L^{\infty}([0,T];H).
\end{align}
Now integrating
\begin{align} \partial_t| v|^2+\nu| v|^2_V\le\left(\frac{2}{\epsilon}| z|^2_V+\frac{\epsilon}{2}\right)| v|^2+\frac{2}{\epsilon}|F(t)|^2,
\end{align}
from 0 to $T$, one gets
\begin{align}\label{aprioriestV}
    | v(T)|^2+\nu\int^T_0| v(t)|^2_Vdt\le \left(\psi_T(z)| v(0)|^2+c_F\right)\int^T_0\left(\frac{2}{\epsilon}| z(t)|^2+\frac{\epsilon}{2}\right)dt+\frac{2}{\epsilon}\int^T_0|F(t)|^2dt+| v(0)|^2.
\end{align}
Which implies
\begin{align}\label{l2x}
     v\in L^2([0,T];V),
\end{align}
and $ v$ is indeed a weak solution. To show that $ v\in C([0,T];H)$, note that $ A: V\to V'$ is bounded and $ A v\in L^2([0,T];V')$. Then $F\in L^2([0,T];V')$ since $ z\in L^4([0,T];L^4(\mathbb{S}^2)\cap H)$ which can be continuously embedded into $V'$, and the terms $B( z),\,B( v, z),\,B( z, v)$ are all in $L^2([0,T];V')$. Combining these facts along with (\ref{l2x}) and invoking lemma 4.1 of \cite{MR3306386}, we conclude that $ v\in C([0,T];H)$.

The uniform apriori estimate (\ref{aprioriestV}) implies that the solution is well defined up to time $t=T_{f,z}$. The iterative process may be repeated starting from $t=T_{f,z}$ with the initial condition $z(t)$. The solution is uniquely extended to $[0,2T_{f,z}]$ and so on to an arbitrary large time.

Now, multiplying both sides of (\ref{mode}) with $Av$, and noting again the classical fact that $\frac{1}{2}\partial_t| v(t)|^2=(\partial_t  v(t), v(t))$ and $( C v, A v)=0$, integrating over $\mathbb{S}^2$, one obtains:
\begin{align*}
(\partial_t  v, A v)+\nu( A v, A v)=-b( v, v, A v)-b( v, z, A v)-b( z, v, A v)+\langle F(t), A v(t)\rangle    
\end{align*}
\begin{align}
    \Longrightarrow\,\frac{1}{2}\frac{d^+}{dt}| v|^2+\nu| A v|^2=-b( v(t), v(t),A v(t))-b( v(t), z(t), A v(t))-b( z(t), v(t), A v(t))+\langle F(t), A v(t)\rangle.
\end{align}
Now,
\begin{align*}
|b( v, v, A v)|\le C| v|^{\frac{1}{2}}| v|_V| A v|^{\frac{3}{2}}\quad\forall\, v\in V,  v\in D(A),
\end{align*}
\begin{align*}
|b( v, z, A v)|\le C| v|^{\frac{1}{2}}| v|^{\frac{1}{2}}_V| z|^{\frac{1}{2}}_V| A v|^{\frac{3}{2}}\quad\forall\, v\in V,  v\in D( A),
\end{align*}
\begin{align*}
|b( z, v, A v)|\le C| z|^{\frac{1}{2}}| z|^{\frac{1}{2}}| v|^{\frac{1}{2}}_V| A v|^{\frac{3}{2}}\quad\forall\, z\in V,  v\in D(A).
\end{align*}
Also,
\begin{align*}
(F(t),Av)&\le\frac{\epsilon}{4}|A v(t)|^2+\frac{1}{\epsilon}|F(t)|^2 .   
\end{align*}
Furthermore, using Young inequality with the choice $p=\frac{4}{3}$ and $ab=(\epsilon p)^{\frac{1}{p}}|Av|^{3/2}\epsilon p)^{-\frac{1}{p}}|v|^{1/2}|v|_V$, the above estimates of the three bilinear terms become:
\begin{align*}
|b( v, v, A v)|&\le C| v|^{\frac{1}{2}}| v|_V| A v|^{\frac{3}{2}}\\
&\le\epsilon| A v|^2+C(\epsilon)| v|^2| v|^4_V,    
\end{align*}
\begin{align*}
|b( v, z, A v)|&\le C| v|^{\frac{1}{2}}| v|^{\frac{1}{2}}_V| z|^{\frac{1}{2}}_V| A v|^{\frac{3}{2}}\\
&\le\epsilon| A v|^2+C(\epsilon)| v|^2| v|^2_V| z|^2_V,
\end{align*}
\begin{align*}
    |b( z, v, A v)|&\le C| z|^{\frac{1}{2}}| z|^{\frac{1}{2}}_V| v|^{\frac{1}{2}}_V| A v|^{\frac{3}{2}}\\
    &\le \epsilon | A v|^2+C(\epsilon)| z|^2| z|^2_V| v|^2_V.
\end{align*}
Therefore,
\begin{align}\label{energyvda}
    \frac{d^+}{dt}| v|^2_V+(2\nu-3\epsilon-\frac{\epsilon}{4})| A v|^2&\le C(\varepsilon)(| v|^2| v|^4_V+| v|^2| v|^2_V| z|^2_V+| z|^2| z|^2_V| v|^2_V)+\frac{1}{\epsilon}|F(t)|^2.
\end{align}
 Momentarily dropping the term
  $| A v(t)|^2$, we have the differential inequality
\begin{align*}
     y'&\le a+\theta  y,\\
 y(t)&=| v|^2_V,\quad a(t)=\frac{1}{\nu}|F(t)|^2,\theta(t)=C(\epsilon)(| v|^2| v|^2_V+| v|^2| z|^2_V+| z|^2| z|^2_V).
\end{align*}
Then for any $\epsilon$ such that $\epsilon<\frac{8}{13}\nu$, using the Gronwall lemma , one has
\begin{align*}
    \frac{d^+}{dt}\left( y(t)e^{-\int^t_0\theta(\tau)d\tau}\right)\le a(t)e^{-\int^t_0\theta(\tau)d\tau}ds
\end{align*}
\begin{align*}
    | v(t)|^2_V&\le| v(0)|^2_V\exp\left(\int^t_0C(\epsilon)(| v(\tau)|^2| v(\tau)|^2_V+| v(\tau)|^2| z(\tau)|^2_V+| z(\tau)|^2| z(\tau)|^2_V)d\tau\right)\\   &+\frac{1}{\nu}\int^t_0|F(s)|^2\exp\left(\int^t_sC(\epsilon)(| v(\tau)|^2| v(\tau)|^2_V+| v(\tau)|^2| z(\tau)|^2_V+| z(\tau)|^2| z(\tau)|^2_V)d\tau\right) ds
\end{align*}

\begin{align}\label{apriori1}
    \sup_{t\in[0,T]}| v(t)|^2_V\le K_3,
\end{align}

\begin{align*}
    K_3=K_3( u_0, F,\nu,T, z)=\left(|v(0)|^2_V+\frac{1}{\nu}\int^T_0|F(s)|^2ds\right)\exp(C(\epsilon)K_2 K_1),
\end{align*}

which implies
\begin{align}\label{linfy}
     v\in L^{\infty}(0,T;V).
\end{align}
Let us now come back to (\ref{energyvda}), which we integrate from $0$ to $T$. After simplifying, we have
\begin{align*}
    \int^T_0|Av(t)|^2 dt\le K_4,
\end{align*}
and
\begin{align*}
    K_4&=K_4( u_0, F,\nu, z,T)\\
&=\frac{1}{2\nu-3\epsilon-\frac{\epsilon}{4}}(| u_0|^2+C(\epsilon)\sup_{t\in[0,T]}|v(t)|^2|v(t)|^4_V+C(\epsilon)\sup_{t\in[0,T]}|v(t)|^2|v(t)|^2_V|z(t)|^2_V\\&+C(\epsilon)\sup_{t\in[0,T]}|z(t)|^2|z(t)|^2_V|v(t)|^2_V +\frac{1}{\epsilon}\int^T_0| F(t)|^2)dt.
\end{align*}
As
\begin{align*}
    \sup_{t\in [0,T]}|v(t)|^2&\le K_2,\\
    \sup_{t\in [0,T]}|v(t)|^4_V&\le K^2_3,\\
    | z(t)|^2_V&\le C_1,\\
    \sup_{t\in[0,T]}| z(t)|^2\le C_2.
\end{align*}
So,
\begin{align*}
K_4=\frac{1}{2\nu-3\epsilon-\frac{\epsilon}{4}}(| u_0|^2+C(\epsilon)K_2K^2_3+C(\epsilon)K_2 K_3 C_1+C(\epsilon)C_2 C_1 K_3 +\frac{1}{\epsilon}\int^T_0| F(t)|^2)dt.
\end{align*}
This implies
\begin{align}
     v\in L^2(0,T_{f,z};D(A)).
\end{align}
It remains to show that $ v\in C([0,T];V).$ Note, the fact that the solution with $ v_0\in V$ is in $L^2([0,T];V)$ implies that a.e. on $[0,T]$, $ v(t)\in V.$ Moreover, since $ v(t)\in C([0,T];H)$ as previously deduced, and is unique as proved in step 1, it follows that $ u\in C([0,T];V)$.

Together with the uniform apriori estimate, the local existence-uniqueness shown in step 1,  allows us to conclude that there exists a unique $ u\in C(0,\infty;H)\cap L^2(0,\infty;V)\subset C(0,\infty;V)\cap L^2(0,\infty;D( A)),$
for any given $ u_0\in V$, $ f\in H$, $ z(t)\in L^4_{\text{loc}}([0,\infty);\mathbb{L}^4(\mathbb{S}^2)\cap H)$. Moreover, our promising apriori bound (\ref{apriori1}) yields $T=\infty.$
\end{proof}
\subsection{Existence and uniqueness of a strong solution with $v_0\in H$}\label{ssec:strongsoln}
\begin{cor}
    If $ f\in H$, $ v_0\in H$, $ z(t)\in L^4_{\text{loc}}([0,\infty);\mathbb{L}^4(\mathbb{S}^2)\cap H)$, then $ v(t)\in V$ for all $t>0.$
\end{cor}
We follow the proof in \cite{MR1207308}.
The idea stems from the standard approximation method commonly used in PDE theory. In view of the apriori estimate (\ref{energyvda}) one takes an approximated solution to (\ref{mild}) in $Y_{T}$. Then one shows the approximates converge. Finally one shows that the limit function indeed satisfies (\ref{mild}).

Let $( v_{0,n})\subset V$ be a sequence converging to $ v_0$ in $H.$ For all $n\in\N$, let $ v_n$ be a solution of equation (\ref{mild}) in $Y_{T}$ corresponding to the intial data $ v_{0,n}.$ Similar to the case when $ v_0\in V$, one can find a constant such that $| v_n|_{X_{T}}\le c,\quad\forall\,n\in\N.$ Following the same lines as in the proofs of (\ref{linfx}) and (\ref{l2x}), $ v_n$ can be proved to be a weak solution.

Moreover, for $n,m\in\N$, take $ v_{n,m}= v_n- v_m$ with $ v^0_{n,m}= v^0_n- v^0_m.$ Then $v_{n,m}$ is the solution of
\begin{align}\label{vnm}
    \begin{cases}
        \partial_t  v_{n,m}+\nu A  v_{n,m}=-B( v_{n,m},z)-B( z, v_{n,m})-B( v_{n,m}, v_n)-B( v_m, v_{n,m})-\mathbf{C}v_{n,m},\\
         v_{n,m}(0)= v^0_n- v^0_m.
    \end{cases}
\end{align}
Multiplying both sides of (\ref{vnm}) with $ v_{n,m}$ and integrating against $ v_{n,m}$, using Lemma \ref{lem1.2} and (\ref{b0}) and noting (\ref{corio}), one obtains
 \begin{align}\label{Cauchy}     \partial_t| v_{n,m}|^2+2\nu| v_{n,m}|^2_V&=-2b( v_{n,m}, z, v_{n,m})-2b( v_{n,m}, v_n,v_{n,m}),\notag\\
\intertext{Since $|b( w, w, z)|\le C| w|| w|_V| z|_V$ and $|b( w, w, v)|\le C| w|| w|_V| v|_V$}
&\le C| v_{n,m}|| v_{n,m}|_V(| z|_V+| v_n|_V)\notag\\
\intertext{Then using the Young inequality with $a=\epsilon| v_{n,m}|_V$ and $b=\frac{C}{\sqrt{\epsilon}}| v_{n,m}|(|z|_V+| v_n|_V),$}
&\le \frac{\epsilon| v_{n,m}|_V}{2}+\frac{C}{2\epsilon}| v_{n,m}|^2(|z|^2_V+| v_n|^2_V).
 \end{align}
Therefore, for any $\epsilon>0$ such that $\frac{\epsilon}{2}<2\nu$, one applies the Gronwall lemma to obtain
\begin{align*}
    \partial_t| v_{n,m}|^2\le \frac{C}{2\epsilon}(| z|^2_V+| v_n|^2_V)| v_{n,m}|^2.
\end{align*}
Combining this with $ v^0_{n,m}= v^0_n- v^0_m$, it is easy to show that
\begin{align*}
    | v_{n,m}(t)|^2\le | v_{n,m}(0)|^2\exp\left(\frac{C}{2\epsilon}(\int_0^{T}| z(t)|^2_V+| v_n(t)|^2_V)| v_{n,m}(t)|^2dt\right)<\infty,
\end{align*}
as $\int^{T}_0| z(t)|^2_V+| v_n(t)|^2_V<\infty$.
Hence $ v_{n,m}$ converges in $T$; and is therefore Cauchy in $T$. That is, for any $\epsilon>0$, $\exists\,N\in\mathbb{N}$ such that $| v_n- v_m|<\epsilon$ whenever $n,m\ge N.$

Let the limit of $ v_n$ be $ v.$
It remains to show $v$ indeed satisfies (\ref{mild}).

Let $v_n$ be the solution to
\begin{align}
     v_n(t)=S(t) v_{0,n}-\int^t_0 S(t-s)(B( u_n(s)))ds+\alpha\int^t_0 z_n(s)ds,
\end{align}
where $ z_n=\int^t_0S(t-s)GdL_n(t).$
We would like to show that
\begin{align}
    \lim_{n\to\infty} u_n(t)=S(t) u_0-\int^t_0 S(t-s)(B( u(s)))ds+\int^t_0S(t-s)fds+\alpha\int^t_0 z(s)ds.
\end{align}
Assume $ f_n\to  f$ in $L^2(0,T;H)$ $ , z_n=\int^t_0S(t-s)GdL_n(t)\to z$ in $L^4([0,T];L^4(\mathbb{S}^2)\cap H),$ we would like to check if
\begin{align}
    \lim_{n\to\infty}B( u_n)=B( u)\quad\text{in}\quad H.
\end{align}
For this, note first that
\begin{align*}
    \left|| u_n|^2_V-| u|^2_V\right|=&\ \left|( u_n, u_n)-( u, u)\right|\\
     =&\ \left|( u_n, u_n)_V-( u, u_n)_V+( u, u_n)_V-( u, u)_V\right|\\
    = &\ \left|( u_n, u_n)_V-( u, u_n)_V\right|+\left|( u, u_n)_V-( u, u)_V\right|\\
    \le &\ | u_n- u|_V| u_n|_V+| u|_V| u_n- u|_V    .
\end{align*}
Now $| u_n|_V$ is Cauchy and is therefore bounded. So $ u_n$ converges to $ u$ in $V$ as $n\to\infty.$
Then using (\ref{b1}) one deduces that
\begin{align*}
&\    |B( u_n)-B( u)|\\
= &\ |B( u_n, u_n)-B( u_n, u)+B( u_n, u)-B( u, u)|\le  C(| u_n|^2_V+| u_n|^2_V| u|+| u|^2_V)\to C| u|^2_V.
\end{align*}
Now analogous to the earlier proof of contraction we have,
\begin{align*}
&\    |B( u_n(s))-B( u(s))|^2_{Y_{T}}\\
\le &\ \left|\int^{t}_0S(t-s)(B( u_n(s))-B( u(s)))ds\right|^2_{Y_{T}}\\
\le &\ c\int^{T}_0|B( u_n(s))-B( u(s))|^2ds\\
\le &\ c| u|^2_{T}\sqrt{T}.
\end{align*}
Therefore, $B( u_n)-B( u)$ is in $L^2(0,T;H).$ Now by the continuity argument again, one has
\begin{align*} \lim_{n\to\infty}\int^{T}_0S(t-s)B( u_n(s))ds=\int^{T}_0S(t-s)B( u(s))ds,
\end{align*}
and
\begin{align*}
  \lim_{n\to\infty}\int^{T}_0S(t-s) f_n(s)ds=\int^{T}_0S(t-s) f(s)ds.
\end{align*}
Combining the above with the assumptions that
\begin{align*}
    \lim_{n\to\infty}S(t) u_{0,n}&=S(t) u_0,\\
    \lim_{n\to\infty} z_n(t)&= z(t),
\end{align*}
one deduces that
\begin{align*}
    \lim_{n\to\infty} v_n(t)= v(t),
\end{align*}
and there exists a solution to (\ref{mild}). However, the solution constructed as the limits of $ u_n$ leaves open the possibility that there are more than one limit. So we will now prove $u$ is unique. The idea is analogous to proving (\ref{Cauchy}). Nevertheless we detail as follows. Suppose $ v_1$, $ v_2$ are two solutions of (\ref{ODE}) with the same initial condition. Let $ w= v_1- v_2$, then $ w$ satisfies
\begin{align}\label{w}
    \begin{cases}
        \partial_t  w+\nu A  w=-B( w, z)-B(z, w)-B( w, v_1)-B( v_2, w),\\
         w(0)=0.
    \end{cases}
\end{align}
Multiplying (\ref{w}) on both sides with $w$ and integrating against $w$, using the identities $\partial_t| v(t)|^2=2\langle \partial_t  v(t), v(t)\rangle$ again in Temam
and (\ref{b0}), one gets
 \begin{align}\label{unique}     \partial_t| w|^2+2\nu| w|^2_V&=-2b( w, z, w)-2b( w, v_1, w),\notag\\
\intertext{Since $|b( w, w, z)|\le C| w|| w|_V| z|_V$ and $|b( w, w, v)|\le C| w|| w|_V| v|_V$}
&\le C| w|| w|_V(| z|_V+| v_1|_V).\notag\\
\intertext{Then via usual Young inequality with $a=\sqrt{\epsilon}| w|_V$ and $b=\frac{C}{\sqrt{\epsilon}}| w|(| z|_V+| v_n|_V)$}
&\le \frac{\epsilon| w|_V}{2}+\frac{C}{2\epsilon}| w|^2(| z|^2_V+| v_1|^2_V).
 \end{align}
Therefore, for any $\epsilon>0$ such that $\frac{\epsilon}{2}<2\nu$, one applies the Gronwall lemma to
\begin{align*}
    \partial_t| w|^2\le \frac{C}{2\epsilon}(| z|^2_V+| v_1|^2_V)| w|^2,
\end{align*}
and combining with $ w_0= v_{1,0}- v_{2,0}=0$, it follows from the Gronwall lemma that
\begin{align*}
    | w(t)|^2\le | w(0)|^2\exp\left(\frac{C}{2\epsilon}(\int_0^{T}| z(t)|^2_V+| v_1(t)|^2_V)| w(t)|^2dt\right)<\infty
\end{align*}
as $\int^{T}_0| z(t)|^2_V+| v_1(t)|^2dt<\infty.$
Now, since $ w(0)=0$, necessarily $ w(t)$ must be 0.

It remains to show $ v\in C((0,T;V)$, as observed from the above energy inequality (\ref{Cauchy}). The solution starts with an initial condition $ v_0\in H$ belonging to
$L^2(0,T;V)$. This implies that almost everywhere in $(0,T]$, there must exist a time point $\epsilon$ (and $\epsilon<T$) such that $u(\epsilon)\in V$. Then one may repeat step two onto another interval $[\epsilon, 2\epsilon],[2\epsilon, 3\epsilon],$ and soon over the whole $[\epsilon,\infty].$  Finally we obtain that $u\in C([\epsilon,T];V)\cap L^2([\epsilon,T];D(A))$ for all $\epsilon>0$. Note that $T=\infty$ as implied from the apriori estimate. 
\\
   
In summary, in this section, we have proved:
\begin{lem}\label{lemfv}
    Assuming that $\alpha\ge 0$, $ z\in L^4_{\text{loc}}([0,\infty);\mathbb{L}^4(\mathbb{S}^2)\cap H)$, $ f\in H$ and $ v_0\in H.$ Then, there exists an unique solution of (\ref{mildv}) in the space $C(0,T;H)\cap L^2(0,T;V),$ which belongs to $C(\epsilon,T; V)\cap L^2_{\text{loc}}(\epsilon,T;D( A))$ for all $\epsilon>0$ and $T>0.$
\end{lem}
\vspace{1cm}
Combining Lemma \ref{lemfv} with \ref{lemfh}, we have proved theorem \ref{t3}.
\begin{rmk}
    Continuous dependence on $ v_0$, $ z$ and $ f$ is implied from the point where local existence and uniqueness is attained and hence holds also for global solutions.
\end{rmk}
\begin{rmk}
    The proof of Theorem \ref{t3} shows that the solution $v$, starting from $v_0\in H$, belongs to $V$ for a.e. $t\ge t_0$. If we take any $\bar{t}\ge t_0$ such that $v(\bar{t})\in V$, the solution is extended over the interval $[t_0, t_0+\epsilon]$ and is found to be in $D(A)$ as well. One may repeat this step over another interval $[t_0+\epsilon,t_0+2\epsilon]$, $[t_0+2\epsilon,t_0+3\epsilon]\cdots.$ Thus, we obtain that $v\in C([t_0+\epsilon,\infty);V)\cap L^2_{\text{loc}}(t_0+\epsilon, D(A)).$
\end{rmk}
Furthermore, provided the noise does not degenerate, based on the condition given in the following, we obtained the existence and uniqueness results for the solution to the original equation (\ref{asnse4}). 

If 
\begin{equation}\label{eq_v}
\sum_{l}\lambda^{\frac{\beta}{2}}|\sigma_l|^{\beta}<\infty,
\end{equation}
then by Lemma \ref{cadlaglem} the process $z$ has a version which has left limits and is right continuous in $V$.
Recall that $u_t:=v_t+z_t$ and for each $T>0$, define
\begin{align}
    Z_T(\omega):=\sup_{0\le t\le T}|z_t(\omega)|_V,\quad\omega\in\Omega.
\end{align}
If \eqref{eq_v} holds then by Lemma \ref{lem_n1} we have 
\[\mathbb EZ_T<\infty.\]
Hence there exists a measurable set $\Omega_0\subset\Omega$ such that $\mathbb P\left(\Omega_0\right)=1$ and 
\[Z_T(\omega)<\infty,\quad \omega\in\Omega_0\,.\]
Finally, let us study (\ref{asnse4}) for $\omega\in\Omega_0$. Since $z(\cdot,\omega)\in D([0,\infty);V)$, it is of course $ z(\cdot,\omega)\in D([0,\infty);H)$. Therefore, by Theorem (\ref{t3}), $u(\cdot,\omega)= v(\cdot,\omega)+ z\cdot(\omega)$ is the unique c\'adl\'ag solution to (\ref{asnse4}).
So, we extend the existence theorem of a strong solution for $ u$.
Moreover, for $\omega\in\Omega_0$  we find that $u(\cdot,\omega)= v(\cdot,\omega)+ z(\cdot,\omega)$ is the unique solution to (\ref{asnse4}) in $D([0,\infty);H)\cap D([0,\infty);V)$ which belongs to $D([h,\infty);V)\cap L^2_{\text{loc}}(h,\infty);D(A))$ for all $h>0.$ If $ u_0\in V$, then $ u\in D([h,\infty);V)\cap L^2_{\text{loc}}([h,\infty);D(A))$ for all $h>0$, $T>0.$

This completes the proof of Theorem \ref{t4}. 

Since the solution is constructed using the Banach Fixed Point Theorem, the continuous dependence on initial data is implied from the existence-uniqueness proof of strong solution in the above line.
Moreover, our existence-uniqueness results work naturally when the initial time $t_0\in \R$ other than 0. 
   
\section{Acknowledgments}
This work is taken out of the author's PhD thesis (awarded in April 2018). The author would like to express her gratitude to her PhD advisor Prof. Ben Goldys for his helpful suggestions to this work. The author is especially grateful to the anonymous referee in pointing out various mistakes. Finally, the author would like to thank Dimitri Levachov for his careful proof-read and various styling suggestions.

\bibliographystyle{plain}
\bibliography{lsthesis}
\end{document}